\documentclass[11pt]{amsart}

\usepackage{amsmath,amsthm,amsfonts,amssymb}
\usepackage[margin=1in]{geometry}
\usepackage[all,cmtip]{xy}
\usepackage{tikz}
\usepackage[bookmarks]{hyperref}
\usepackage{cite,verbatim,enumitem}
\usepackage[textwidth=.8in]{todonotes}

\newenvironment{enumalph}{\begin{enumerate}  }{\end{enumerate}}

\newenvironment{enumroman}{\begin{enumerate}  }{\end{enumerate}}

\newcommand{\relphantom}[1]{\mathrel{\phantom{#1}}}

\newcommand{\set}[1]{\left\{ #1 \right\}}
\newcommand{\abs}[1]{\left| #1 \right|}

\newcommand{\wt}[1]{\widetilde{ #1}}
\newcommand{\wh}[1]{\widehat{ #1 }}
\newcommand{\ol}[1]{\overline{#1}}

\newcommand{\ve}{\varepsilon}

\newcommand{\wtalpha}{\wt{\alpha}}

\DeclareMathOperator{\coker}{coker}

\DeclareMathOperator{\ev}{ev}
\DeclareMathOperator{\Ext}{Ext}

\DeclareMathOperator{\opH}{H}
\newcommand{\Hbul}{\opH^\bullet}
\DeclareMathOperator{\hd}{hd}
\DeclareMathOperator{\Hom}{Hom}

\DeclareMathOperator{\id}{id}

\DeclareMathOperator{\ind}{ind}
\DeclareMathOperator{\Lie}{Lie}
\renewcommand{\mod}{\;\textup{mod}\;}

\DeclareMathOperator{\rad}{rad}

\DeclareMathOperator{\res}{res}
\DeclareMathOperator{\soc}{soc}

\DeclareMathAlphabet{\mathpzc}{OT1}{pzc}{m}{it}

\newcommand{\F}{\mathbb{F}}

\newcommand{\Z}{\mathbb{Z}}

\newcommand{\cC}{\mathcal{C}}

\newcommand{\cG}{\mathcal{G}}

\newcommand{\cgr}{\cG_r}

\newcommand{\g}{\mathfrak{g}}

\newcommand{\fu}{\mathfrak{u}}

\newcommand{\Fp}{\F_p}
\newcommand{\Fq}{\F_q}

\newcommand{\Czbar}{\ol{C}_{\Z}}

\newcommand{\Gfp}{G(\Fp)}
\newcommand{\Gfq}{G(\Fq)}

\newcommand{\Tfp}{T(\Fp)}
\newcommand{\Tfq}{T(\Fq)}

\numberwithin{equation}{subsection}

\newtheorem{theorem}{Theorem}[subsection]
\newtheorem{proposition}[theorem]{Proposition}

\newtheorem{corollary}[theorem]{Corollary}
\newtheorem{lemma}[theorem]{Lemma}

\newtheorem*{theorem*}{Theorem}
\newtheorem*{lemma*}{Lemma}

\theoremstyle{definition}

\newtheorem{remark}[theorem]{Remark}

\title[Second Cohomology for Finite Groups of Lie Type]{Second Cohomology for Finite Groups of Lie Type}

\thanks{The members of the 2010--2011 University of Georgia VIGRE Algebra Group are Brian D.\ Boe, Brian Bonsignore, Theresa Brons, Jon F.\ Carlson, Leonard Chastkofsky, Christopher M.\ Drupieski, Niles Johnson, Wenjing Li, Phong Thanh Luu, Tiago Macedo, Daniel K.\ Nakano, Nham Vo Ngo, Brandon L.\ Samples, Andrew J.\ Talian, Lisa Townsley, and Benjamin J.\ Wyser.}

\author{University of Georgia VIGRE Algebra Group}

\address{
Department of Mathematics \\
University of Georgia \\
Athens, GA~30602-7403}

\subjclass[2010]{Primary 20G10, 20C33; Secondary 20G05, 20J06.}

\begin{document}

\begin{abstract}
Let $G$ be a simple, simply-connected algebraic group defined over $\Fp$. Given a power $q = p^r$ of $p$, let $\Gfq \subset G$ be the subgroup of $\Fq$-rational points. Let $L(\lambda)$ be the simple rational $G$-module of highest weight $\lambda$. In this paper we establish sufficient criteria for the restriction map in second cohomology $\opH^2(G,L(\lambda)) \rightarrow \opH^2(\Gfq,L(\lambda))$ to be an isomorphism. In particular, the restriction map is an isomorphism under very mild conditions on $p$ and $q$ provided $\lambda$ is less than or equal to a fundamental dominant weight. Even when the restriction map is not an isomorphism, we are often able to describe $\opH^2(\Gfq,L(\lambda))$ in terms of rational cohomology for $G$. We apply our techniques to compute $\opH^2(\Gfq,L(\lambda))$ in a wide range of cases, and obtain new examples of nonzero second cohomology for finite groups of Lie type.
\end{abstract}

\maketitle

\section{Introduction}

\subsection{}

Let $k$ be an algebraically closed field of characteristic $p > 0$, and let $G$ be a simple, simply-connected algebraic group defined over the prime field $\Fp$. Given a power $q = p^r$ of $p$, let $\Gfq \subset G$ be the subgroup of $\Fq$-rational points. In their famous 1977 paper \cite{Cline:1977a}, Cline, Parshall, Scott, and van der Kallen related the cohomology of the finite group $\Gfq$ to the rational cohomology of the ambient algebraic group $G$. Specifically, given a finite-dimensional rational $G$-module $M$ and an integer $n \geq 0$, they showed that for all sufficiently large integers $e$ and $f$, the restriction map $\opH^n(G,M^{(e)}) \rightarrow \opH^n(G(\F_{p^{e+f}}),M^{(e)})$ is an isomorphism \cite[Theorem 6.6]{Cline:1977a}. Here $M^{(e)}$ denotes the rational $G$-module obtained by twisting the structure map for $M$ by the $e$-th iterate of the Frobenius morphism $F: G \rightarrow G$. The Frobenius morphism restricts to an automorphism of $\Gfq$, so one gets $\opH^n(\Gfq,M^{(e)}) \cong \opH^n(\Gfq,M)$ for all $e \geq 0$. Thus, the aforementioned theorem asserts that for sufficiently large $q$, the space $\opH^n(\Gfq,M)$ stabilizes to a fixed value. This stable value is called the \emph{generic cohomology} of $G$ in $M$. Unfortunately, even for small values of $n$, generic cohomology is unsatisfactory as a tool for computing $\opH^n(\Gfq,M)$ directly in terms of $\opH^n(G,M)$.

In this paper we develop new techniques for computing cohomology for the finite group $\Gfq$ directly in terms of cohomology for the ambient algebraic group $G$, and apply those techniques in the case when $M$ is a simple $\Gfq$-module. For $n \leq 2$ we establish specific conditions under which the restriction map $\opH^n(G,L(\lambda)) \rightarrow \opH^n(\Gfq,L(\lambda))$ is an isomorphism. Using this isomorphism, we are then able to compute the second cohomology group $\opH^2(\Gfq,L(\lambda))$ provided $\lambda$ is less than or equal to a fundamental dominant weight, and assuming some additional restrictions when the underlying root system is of type $C_n$. Even in cases where our techniques fail to directly yield an isomorphism $\opH^2(G,L(\lambda)) \cong \opH^2(\Gfq,L(\lambda))$, for example, when $\lambda$ is the highest long root or the highest short root of the root system of $G$, we are often able to identify $\opH^2(\Gfq,L(\lambda))$ with a rational cohomology group for $G$. In this way we obtain new nonzero calculations of second cohomology for finite groups of Lie type when the coefficients are taken in the adjoint representation. A salient feature of our results is that we require no twisting of the coefficient module by the Frobenius morphism, which enables us to make calculations for relatively small values of $p$ and $q$.

The machinery developed and calculations performed here are interesting in several respects. First, our calculations extend to second cohomology the seminal first cohomology calculations of Cline, Parshall, and Scott \cite{Cline:1975,Cline:1977b} and of Jones \cite{Jones:1975}, who computed for all $p$ and for almost all $q$ the space $\opH^1(\Gfq,L(\lambda))$ when $\Gfq$ is a (twisted or untwisted) finite Chevalley group and $\lambda$ is a minimal nonzero dominant weight. We obtain for $p \geq 5$ new proofs of the second cohomology calculations made by Bell \cite{Bell:1978} for finite special linear groups, and confirm many of the calculations made by Avrunin \cite{Avrunin:1978} and Kleshchev \cite{Kleshchev:1994}. Second, the machinery developed here generalizes and simplifies the techniques we developed in our previous paper \cite{UGA:2011} for the case $n=1$. Third, our machinery introduces the infinitesimal Frobenius kernels $G_r$, $B_r$, and $U_r$ into the study of the restriction map $\opH^n(G,M) \rightarrow \opH^n(\Gfq,M)$. This is in contrast to the works cited above, in which the authors made no use of infinitesimal group schemes.

\subsection{Notation, Organization, and Main Results}

The notation used in this paper is the same as in \cite{UGA:2011}, and uses the conventions stated in \cite{Jantzen:2003}. In particular, we use the standard Bourbaki labeling of simple roots and fundamental dominant weights (cf.\ also \cite[\S 13]{Humphreys:1978}).

The paper is organized as follows: In Section \ref{section:comparingalgebraicandfinite} we study the induction functor $\ind_{\Gfq}^G(-)$ from the category of $k\Gfq$-modules to the category of rational $G$-modules, and use it to obtain a long exact sequence in cohomology containing for each $n \geq 0$ the restriction map $\opH^n(G,M) \rightarrow \opH^n(\Gfq,M)$. By analyzing the terms in the long exact sequence we establish the following:

\begin{theorem} \label{theorem:vanishingconditionforExt1iso}
Let $\lambda \in X_r(T)$. Suppose that $\Ext_{U_r}^1(k,L(\lambda))$ is semisimple as a $B/U_r$-module, and $\Ext_{U_r}^1(k,L(\lambda))^{\Tfq} = \Ext_{U_r}^1(k,L(\lambda))^T$. Then the restriction maps
\[
\opH^1(G,L(\lambda)) \rightarrow \opH^1(\Gfq, L(\lambda)) \quad \text{and} \quad \opH^2(G,L(\lambda)) \rightarrow \opH^2(\Gfq,L(\lambda))
\]
are an isomorphism and an injection, respectively.
\end{theorem}

\begin{theorem} \label{theorem:Ext2Gfqiso}
Let $\lambda \in X_r(T)$. Suppose that $\Ext_{U_r}^1(k,L(\lambda))$ is semisimple as a $B/U_r$-module, $\Ext_{U_r}^i(k,L(\lambda))^{\Tfq} = \Ext_{U_r}^i(k,L(\lambda))^T$ for $i \in \set{1,2}$, and
\[
p^r > \max \{-(\nu,\gamma^\vee): \gamma \in \Delta, \nu \in X(T),\Ext_{U_r}^1(k,L(\lambda))_\nu \neq 0 \}.
\]
Then the restriction maps
\[
\opH^2(G,L(\lambda)) \rightarrow \opH^2(\Gfq,L(\lambda)) \quad \text{and} \quad \opH^3(G,L(\lambda)) \rightarrow \opH^3(\Gfq,L(\lambda))
\]
are an isomorphism and an injection, respectively.
\end{theorem}

Theorems \ref{theorem:vanishingconditionforExt1iso} and \ref{theorem:Ext2Gfqiso} are established in Sections \ref{subsection:proofofvanishingconditionforExt1iso} and \ref{subsection:proofofExt2Gfqiso}, respectively. In Section \ref{subsection:alternatecalculation} we apply techniques developed in \cite{Bendel:2011a} to provide a characterization of $\opH^2(\Gfq,L(\lambda))$ in certain cases where the restriction map $\opH^2(G,L(\lambda)) \rightarrow \opH^2(\Gfq,L(\lambda))$ need not be an isomorphism.

In Section \ref{section:Urcohomology} we study the cohomology group $\Ext_{U_r}^1(L(\lambda),k)$, showing under mild restrictions on $p$ that it is semisimple as a $B/U_r$-module provided $\lambda \in X(T)_+$ is a dominant root or is less than or equal to a fundamental dominant weight. A list of the dominant roots, that is, the highest long and short roots in $\Phi$, written in terms of the fundamental weight basis, is provided in Table \ref{table:problemweights}. For a list of all $\lambda \in X(T)_+$ that are less than or equal to a fundamental dominant weight, consult \cite[\S 7.1]{UGA:2011}. In particular, if the underlying root system is of classical type and $\lambda \in X(T)_+$, then $\lambda \leq \omega_j$ only if $\lambda = 0$ or $\lambda = \omega_i$ for some $1 \leq i \leq j$. The results obtained in Section \ref{section:Urcohomology} generalize to arbitrary $r \geq 1$ those obtained in \cite[\S 3]{UGA:2011} for the special case $r=1$.

In Section \ref{section:applications} we apply the explicit calculation of $\Ext_{U_r}^1(L(\lambda),k)$ to show, with only a few exceptions, that the hypotheses of Theorems \ref{theorem:vanishingconditionforExt1iso} and \ref{theorem:Ext2Gfqiso} are satisfied when $\lambda \in X(T)_+$ is less than or equal to a fundamental dominant weight and when $p$ and $q$ satisfy the restrictions in Table \ref{table:pandq}. In particular, we obtain the following theorem, which is established in Section \ref{subsection:proofofExt2restrictioniso}:

\begin{theorem} \label{theorem:Ext2restrictioniso}
Let $\lambda \in X(T)_+$ be less than or equal to a fundamental dominant weight. Suppose $p$ and $q$ satisfy the restrictions stated in Table \ref{table:pandq}, and $\lambda$ is not one of the weights listed in Table~\ref{table:problemweights}. Then the restriction map
\[
\opH^{2}(G,L(\lambda)) \rightarrow \opH^{2}(\Gfq,L(\lambda))
\]
is an isomorphism.
\end{theorem}

\begin{table}[htbp]
\renewcommand{\arraystretch}{1.2}
\begin{tabular}{ll}
\hline
Type \hspace{2ex} & Restrictions on $p$ and $q$ \\
\hline
$A_n$ & $p > 3$ \\
$B_n$ & $p > 3$ \\
$C_n$ & $p > 3$ ($q > 5$ if $\lambda \in \Z\Phi$) \\
$D_n$ & $p > 3$ \\
$E_6$ & $p > 3$ \\
$E_7$ & $p > 3$, $q > 5$ \\
$E_8$ & $p > 5$ \\
$F_4$ & $p > 3$, $q > 5$ \\
$G_2$ & $p > 5$ \\
\hline
& \\
\end{tabular}
\caption{Restrictions on $p$ and $q$.} \label{table:pandq}
\end{table}

\begin{table}[htbp]
\renewcommand{\arraystretch}{1.2}
\begin{tabular}{ll}
\hline
Type & Weights \\
\hline
$A_2$, $q=5$ & $\omega_1,\omega_2$ \\
$A_n$, $n \geq 1$ & $\wtalpha = \omega_1 + \omega_n$ \\
$B_2$ & $\alpha_0 = \omega_1$, $\wtalpha = 2\omega_2$ \\
$B_n$, $n \geq 3$ & $\alpha_0 = \omega_1$, $\wtalpha = \omega_2$ \\
$C_n$, $n \geq 3$ & $\alpha_0 = \omega_2$, $\wtalpha = 2\omega_1$ \\
$D_n$, $n \geq 4$ & $\wtalpha = \omega_2$ \\
$E_6$ & $\wtalpha = \omega_2$ \\
$E_7$ & $\wtalpha = \omega_1$ \\
$E_8$ & $\wtalpha = \omega_8$ \\
$F_4$ & $\alpha_0 = \omega_4,\wtalpha = \omega_1$ \\
$G_2$ & $\alpha_0 = \omega_1,\wtalpha = \omega_2$ \\
\hline
& \\
\end{tabular}
\caption{Certain fundamental weights and dominant roots. Highest short roots are denoted by $\alpha_0$, and highest long roots are denoted by $\wtalpha$.} \label{table:problemweights}
\end{table}

In Sections \ref{subsection:vanishingresults}--\ref{subsection:E8p31} we summarize known results on the cohomology of rational $G$-modules, and obtain in Section \ref{subsection:proofofH2vanishing}, as an immediate corollary of Theorem \ref{theorem:Ext2restrictioniso}, the following calculation of second cohomology groups for $\Gfq$:

\begin{corollary} \label{corollary:H2vanishing}
Let $\lambda \in X(T)_+$ be less than or equal to a fundamental dominant weight. Suppose $p$ and $q$ satisfy the restrictions stated in Table \ref{table:pandq}, and that $\lambda$ is not one of the weights listed in Table~\ref{table:problemweights}. If $\Phi$ is of type $C_n$ and $\lambda = \omega_j$ with $j$ even, assume that $p > n$, and if $\Phi$ is of type $E_8$ with $p=31$, assume $\lambda \neq \omega_7+\omega_8$. Then $\opH^2(\Gfq,L(\lambda)) =0$, except possibly for the following cases:
\begin{itemize}
\item $\Phi$ is of type $E_7$, $p=5$, and $\lambda = 2\omega_7$;
\item $\Phi$ is of type $E_7$, $p=7$, and $\lambda = \omega_2+\omega_7$;
\item $\Phi$ is of type $E_8$, $p=7$, and $\lambda \in \set{2\omega_7,\omega_1+\omega_7,\omega_2+\omega_8}$;
\item $\Phi$ is of type $E_8$, $p=31$, and $\lambda =\omega_6+\omega_8$.
\end{itemize}
If $\Phi$ is of type $E_{8}$ and $p=31$, then $\opH^2(\Gfq,L(\omega_7+\omega_8)) \cong k$. If the Lusztig Character Formula\footnote{The Lusztig Character Formula is conjectured to hold for semisimple algebraic groups when $p\geq h$.} holds for type $E_8$ when $p=31$, then $\opH^2(\Gfq,L(\omega_{6}+\omega_{8}))=0$.
\end{corollary}

In Section \ref{subsection:additionalcalculations} we consider the second cohomology groups $\opH^2(\Gfq,L(\lambda))$ for the dominant weights listed in Table \ref{table:problemweights}, and obtain the following theorem:

\begin{theorem} \label{theorem:otherH2calculations}
Suppose $p$ and $q$ satisfy the conditions listed in Table \ref{table:pandq}, and let $\lambda$ be one of the weights listed in Table \ref{table:problemweights}. Assume also that the following conditions hold:
\begin{itemize}
\item If $\Phi$ is of type $A_n$ and $\lambda = \wtalpha$, assume that $p$ does not divide $n+1$.
\item If $\Phi$ is of type $B_n$ and $\lambda = \wtalpha$, assume that $p$ does not divide $n-1$.
\item If $\Phi$ is of type $C_n$ and $\lambda = \alpha_0$, assume that $p$ does not divide $n$.
\item If $\Phi$ is of type $A_3$ or $B_2$, assume that $q > 5$.
\end{itemize}
Then $\opH^2(\Gfq,L(\lambda)) = 0$ if $\lambda = \alpha_0 \neq \wtalpha$, and is one-dimensional otherwise.
\end{theorem}

Under the assumptions of Theorem \ref{theorem:otherH2calculations} one has $L(\lambda) = H^0(\lambda)$ for each weight in Table \ref{table:problemweights}, and hence $\opH^2(G,L(\lambda)) = \Ext_G^2(V(0),H^0(\lambda)) = 0$. Thus, each nonzero cohomology group in Theorem \ref{theorem:otherH2calculations} provides an example for which $\opH^2(\Gfq,L(\lambda)) \not\cong \opH^2(G,L(\lambda))$. The nonzero second cohomology computed by Theorem \ref{theorem:otherH2calculations} for type $A_2$ when $q=5$ was previously observed by Bell \cite{Bell:1978}, but the other nonzero second cohomology groups found by Theorem \ref{theorem:otherH2calculations} appear to be new. For $\lambda = \wtalpha$, the assumptions on $p$ imply that $L(\wtalpha)$ is the adjoint representation $\g$. Then Theorem \ref{theorem:otherH2calculations} is related, via generic cohomology, to McNinch's results for $\opH^2(G,\g^{(d)})$ \cite[Theorem C]{McNinch:2002}.

Finally, in Section \ref{section:typeC} we summarize our partial results for $\opH^2(G,L(\lambda))$ when the underlying root system is of type $C_n$ and $\lambda$ is a fundamental dominant weight in the root lattice (i.e., a fundamental dominant weight indexed by an even integer). Our results make heavy use of Adamovich's \cite{Adamovich:1986} combinatorial description, as presented in the paper of Kleshchev and Sheth \cite{Kleshchev:1999}, for the submodule lattice of Weyl modules over the symplectic group having fundamental highest weight. Assuming the conditions on $p$ and $q$ in Table \ref{table:pandq}, our analysis for type $C_n$ provides for $p=5$ a large number of additional examples of nonzero second cohomology groups for finite groups of Lie type. It is interesting that for every example we have been able to explicitly compute (in any Lie type), the dimension of $\opH^2(\Gfq,L(\lambda))$ is  always at most one. Some open questions and directions for future research are discussed in Section \ref{subsection:openquestions}.

\section{Comparing algebraic and finite group cohomology} \label{section:comparingalgebraicandfinite}

We begin our investigation of finite group cohomology by considering the exact induction functor $\cgr(-) = \ind_{\Gfq}^G(-)$ from the category of $k\Gfq$-modules to the category of rational $G$-modules. Given an irreducible $G$-module $L(\lambda)$, our goal is to establish sufficient conditions under which the restriction map $\opH^i(G,L(\lambda)) \rightarrow \opH^i(\Gfq,L(\lambda))$ is an isomorphism. The functor $\cgr(-)$ enables us to reduce this problem to one of showing that certain rational cohomology groups for $G$ are zero. A truncated version of the functor $\cgr(-)$ was first introduced by Bendel, Nakano, and Pillen in \cite{Bendel:2001}, where it was used to compare the cohomology theories for the full algebraic group $G$, the finite group $\Gfq$, and the infinitesimal Frobenius kernel $G_r$. The functor $\cgr(-)$ was also employed in \cite{Bendel:2002,Bendel:2004a} to investigate self-extensions and first cohomology, and more recently in \cite{Bendel:2011a} to locate the first non-trivial cohomology class in the cohomology ring $\Hbul(\Gfq,k)$. 

\subsection{The long exact sequence for restriction} \label{subsection:longexactsequence}

Set $\cgr(-) = \ind_{\Gfq}^G(-)$. Let $\iota: k \rightarrow \cgr(k)$ be the homomorphism induced by Frobenius reciprocity from the identity map $\id: k \rightarrow k$, and set $N = \coker(\iota)$. Then there exists a short exact sequence of $G$-modules
\begin{equation} \label{eq:Grkses}
0 \rightarrow k \stackrel{\iota}{\rightarrow} \cgr(k) \rightarrow N \rightarrow 0.
\end{equation}
Let $M$ be a rational $G$-module. By the tensor identity \cite[I.3.6]{Jantzen:2003}, $M \otimes \cgr(k) \cong \cgr(M)$. Then applying the exact functor $M \otimes -$ to \eqref{eq:Grkses}, one obtains the new short exact sequence
\begin{equation} \label{eq:sesM}
0 \rightarrow M \rightarrow \cgr(M) \rightarrow M \otimes N \rightarrow 0,
\end{equation}
and hence the associated long exact sequence in cohomology
\begin{equation} \label{eq:lesM}
\begin{array}{cclclclcc}
0 & \longrightarrow & \Hom_G(k,M) & \longrightarrow & \Hom_G(k,\cgr(M)) &\longrightarrow & \Hom_G(k,M \otimes N) &&\\
& \longrightarrow & \Ext_G^1(k,M) & \longrightarrow & \Ext_G^1(k,\cgr(M)) & \longrightarrow & \Ext_G^1(k,M \otimes N) &&\\
& \longrightarrow & \Ext_G^2(k,M) & \longrightarrow & \Ext_G^2(k,\cgr(M)) & \longrightarrow & \Ext_G^2(k,M \otimes N) & \longrightarrow & \cdots \\
\end{array}
\end{equation}
Since the subgroup $\Gfq$ of $G$ is finite, the quotient $G/\Gfq$ is affine, so the induction functor $\cgr(-)$ is exact by \cite[I.5.13]{Jantzen:2003}. Then by generalized Frobenius reciprocity \cite[I.4.6]{Jantzen:2003}, there exists for each $i \geq 0$ an isomorphism
\begin{equation} \label{eq:generalizedfrobenius}
\Ext_G^i(k,\cgr(M)) \cong \Ext_{\Gfq}^i(k,M).
\end{equation}
Explicitly, the isomorphism \eqref{eq:generalizedfrobenius} is realized by the composition
\[
\Ext_G^i(k,\cgr(M)) \rightarrow \Ext_{\Gfq}^i(k,\cgr(M)) \rightarrow \Ext_{\Gfq}^i(k,M),
\]
where the first arrow is the restriction map induced by the inclusion $\Gfq \subset G$, and the second arrow is induced by the evaluation homomorphism $\ev: \cgr(M) \rightarrow M$, which is a homomorphism of $\Gfq$-modules. Since the map $M \rightarrow \cgr(M)$ in \eqref{eq:sesM} is induced by Frobenius reciprocity from the identity $\id: M \rightarrow M$, it then follows from the explicit description of \eqref{eq:generalizedfrobenius} that the maps $\Ext_G^i(k,M) \rightarrow \Ext_G^i(k,\cgr(M))$ in \eqref{eq:lesM} identify with the cohomological restriction maps
\[
\res: \Ext_G^i(k,M) \rightarrow \Ext_{\Gfq}^i(k,M).
\]
Then the long exact sequence \eqref{eq:lesM} may be rewritten in the form
\begin{equation} \label{eq:lesrestriction}
\begin{array}{cclclclcc}
0 & \longrightarrow & \Hom_G(k,M) & \stackrel{\res}{\longrightarrow} & \Hom_{\Gfq}(k,M) &\longrightarrow & \Hom_G(k,M \otimes N) &&\\
& \longrightarrow & \Ext_G^1(k,M) & \stackrel{\res}{\longrightarrow} & \Ext_{\Gfq}^1(k,M) & \longrightarrow & \Ext_G^1(k,M \otimes N) &&\\
& \longrightarrow & \Ext_G^2(k,M) & \stackrel{\res}{\longrightarrow} & \Ext_{\Gfq}^2(k,M) & \longrightarrow & \Ext_G^2(k,M \otimes N) & \longrightarrow & \cdots \\
\end{array}
\end{equation}
In particular, the restriction map $\Ext_G^i(k,M) \rightarrow \Ext_{\Gfq}^i(k,M)$ is an isomorphism whenever both $\Ext_G^{i-1}(k,M \otimes N) = 0$ and $\Ext_G^{i}(k,M \otimes N) = 0$.

\subsection{Analyzing terms in the long exact sequence}

The coordinate ring $k[G]$ is naturally a $G \times G$-module via the left and right regular actions, respectively. Let $\phi: G \rightarrow G \times G$ be the group homomorphism defined by $\phi(g) = (g,F^r(g))$, and let $k[G]^\vee$ denote the $G$-module obtained by restricting the $G \times G$-action on $k[G]$ to $G$ via $\phi$. In \cite[Proposition 2.4]{Bendel:2011a} it is shown that $\cgr(k) \cong k[G]^\vee$ as $G$-modules, and hence that $\cgr(k)$ admits a filtration by $G$-submodules with sections of the form $H^0(\mu) \otimes H^0(\mu^*)^{(r)}$ for $\mu \in X(T)_+$, with each section occurring exactly once. Then the module $M \otimes N$ appearing in \eqref{eq:sesM} admits a filtration with sections of the form $M \otimes H^0(\mu) \otimes H^0(\mu^*)^{(r)}$ for $0 \neq \mu \in X(T)_+$, each occurring once. It follows that if
\begin{equation} \label{eq:vanishingcondition}
\Ext_G^i(k,M \otimes H^0(\mu) \otimes H^0(\mu^*)^{(r)}) = 0 \quad \text{for all $0 \neq \mu \in X(T)_+$,}
\end{equation}
then $\Ext_G^i(k,M \otimes N) =0$. Given a module $M$, we will be interested in values for $r$ and $i$ such that the condition \eqref{eq:vanishingcondition} is satisfied. To analyze the $\Ext$-group in \eqref{eq:vanishingcondition}, we will consider the Lyndon--Hochschild--Serre (LHS) spectral sequence
\begin{align*}
E_2^{i,j} &= \Ext_{G/G_r}^i(k,\Ext_{G_r}^j(k,M \otimes H^0(\mu) \otimes H^0(\mu^*)^{(r)})) \\
&\Rightarrow \Ext_G^{i+j}(k,M \otimes H^0(\mu) \otimes H^0(\mu^*)^{(r)}),
\end{align*}
which can be rewritten as
\begin{equation} \label{eq:LHSforvanishing}
E_2^{i,j} = \Ext_{G/G_r}^i(V(\mu)^{(r)},\Ext_{G_r}^j(k,M \otimes H^0(\mu))) \Rightarrow \Ext_G^{i+j}(V(\mu)^{(r)},M \otimes H^0(\mu)),
\end{equation}
and the associated five-term exact sequence
\begin{equation} \label{eq:LHSforvanishing5term}
0 \rightarrow E_2^{1,0} \rightarrow \Ext_G^1(V(\mu)^{(r)},M \otimes H^0(\mu)) \rightarrow E_2^{0,1} \rightarrow E_2^{2,0} \rightarrow \Ext_G^2(V(\mu)^{(r)},M \otimes H^0(\mu)).
\end{equation}

\begin{lemma} \label{lemma:E2i0vanishing}
Let $\mu \in X(T)_+$, let $\lambda \in X_r(T)$, and set $M = L(\lambda)$. Then in the spectral sequence \eqref{eq:LHSforvanishing}, $E_2^{i,0} = 0$ for all $i \geq 1$. If also $\mu \neq 0$, then $E_2^{0,0} = 0$, and
\[
\Hom_G(k,L(\lambda) \otimes H^0(\mu) \otimes H^0(\mu^{*})^{(r)}) = 0.
\]
\end{lemma}

\begin{proof}
It follows from \cite[I.6.12]{Jantzen:2003} and the tensor identity that $E_2^{i,0}$ may be rewritten as
\begin{align*}
E_2^{i,0} &= \Ext_{G/G_r}^i(V(\mu)^{(r)},\Hom_{G_r}(k,L(\lambda) \otimes H^0(\mu))) \\
& \cong \Ext_{G/G_r}^i(V(\mu)^{(r)},\ind_{B/B_r}^{G/G_r} \Hom_{B_r}(k,L(\lambda) \otimes \mu)).
\end{align*}
Write $\mu = \mu_0 + p^r \mu_1$ with $\mu_0 \in X_r(T)$ and $\mu_1 \in X(T)_+$. Since $\lambda \in X_r(T)$, there exists a $B/B_r$-module isomorphism
\[
\Hom_{B_r}(k,L(\lambda) \otimes \mu) \cong \begin{cases} 0 & \text{if } \mu_0 \neq -w_0\lambda, \\ p^r\mu_1 & \text{if } \mu_0 = -w_0\lambda. \end{cases}
\]
Suppose $\mu_0 = -w_0\lambda$. Then
\begin{equation} \label{eq:untwistExt}
\begin{split}
E_2^{i,0} \cong \Ext_{G/G_r}^i(V(\mu)^{(r)},\ind_{B/B_r}^{G/G_r} p^r\mu_1) &\cong \Ext_{G/G_r}^i(V(\mu)^{(r)},H^0(\mu_1)^{(r)}) \\
&\cong \Ext_G^i(V(\mu),H^0(\mu_1)).
\end{split}
\end{equation}
The last term in \eqref{eq:untwistExt} is nonzero only if $i=0$ and $\mu = \mu_1$ by \cite[II.4.13]{Jantzen:2003}, so $E_2^{i,0}=0$ for all $i \geq 1$. This proves the first statement of the lemma. Now suppose $\mu = \mu_1$. Then $\mu_0 = \mu - p^r\mu_1 = -(p^r-1)\mu$. Since $\mu,\mu_0 \in X(T)_+$, this is possible only if $\mu = \mu_0 = 0$. Thus, if $\mu \neq 0$, then $E_2^{0,0}=0$. Since $E_2^{0,0} \cong E_\infty^{0,0} \cong \Hom_G(k,L(\lambda) \otimes H^0(\mu) \otimes H^0(\mu^{*})^{(r)})$, this proves the lemma.
\end{proof}

We can now provide a new proof of a well-known result on the restriction map for first cohomology, which was first proved by Cline, Parshall, Scott, and van der Kallen:

\begin{corollary} \textup{(cf.\ \cite[Theorem 7.4]{Cline:1977a})} \label{corollary:CPSK7.4}
Let $\lambda \in X_r(T)$. Then the restriction map
\[
\res: \opH^1(G,L(\lambda)) \rightarrow \opH^1(\Gfq,L(\lambda))
\]
is injective.
\end{corollary}

\begin{proof}
By Lemma \ref{lemma:E2i0vanishing}, the vanishing condition \eqref{eq:vanishingcondition} is satisfied for $i=0$ and $M = L(\lambda)$. Then $\Hom_G(k,L(\lambda) \otimes N) = 0$ in the long exact sequence \eqref{eq:lesrestriction}, which implies the result.
\end{proof}

\subsection{The restriction map for first cohomology} \label{subsection:restrictionmapforH1}

Fix $\lambda \in X_r(T)$, and set $M = L(\lambda)$. In this section we investigate conditions under which the term $\Ext_G^1(k,M \otimes N)$ in \eqref{eq:lesrestriction} is zero, and hence the restriction map $\opH^1(G,L(\lambda)) \rightarrow \opH^1(\Gfq,L(\lambda))$ is an isomorphism.

\begin{lemma} \label{lemma:ext1isotohom}
Let $\mu \in X(T)_+$ and $\lambda \in X_r(T)$. Then
\[
\Ext_G^1(V(\mu)^{(r)},L(\lambda) \otimes H^0(\mu)) \cong \Hom_{G/G_r}(V(\mu)^{(r)},\Ext_{G_r}^1(k,L(\lambda) \otimes H^0(\mu))).
\]
\end{lemma}

\begin{proof}
By Lemma \ref{lemma:E2i0vanishing}, the terms $E_2^{1,0}$ and $E_2^{2,0}$ in the spectral sequence \eqref{eq:LHSforvanishing} with $M = L(\lambda)$ are zero. Then the five-term exact sequence \eqref{eq:LHSforvanishing5term} collapses to yield the stated isomorphism.
\end{proof}

\begin{lemma} \label{lemma:goodfiltrationforExt1}
Let $\mu \in X(T)_+$ and $\lambda \in X_r(T)$. Suppose $\Ext_{U_r}^1(k,L(\lambda))$ is semisimple as a $B/U_r$-module. Then $\Ext_{G_r}^1(k,L(\lambda) \otimes H^0(\mu))^{(-r)}$ is isomorphic to a direct sum of induced modules. In particular, $\Ext_{G_r}^1(k,L(\lambda) \otimes H^0(\mu))^{(-r)}$ admits a good filtration.
\end{lemma}

\begin{proof}
By \cite[II.12.8]{Jantzen:2003}, there exists a natural isomorphism of $G/G_r$-modules
\[
\Ext_{G_r}^1(k,L(\lambda) \otimes H^0(\mu)) \cong \ind_{B/B_r}^{G/G_r} \Ext_{B_r}^1(k,L(\lambda) \otimes \mu).
\]
Also, $\Ext_{B_r}^1(k,L(\lambda) \otimes \mu) \cong \Ext_{U_r}^1(k,L(\lambda) \otimes \mu)^{T_r} \cong (\Ext_{U_r}^1(k,L(\lambda)) \otimes \mu)^{T_r}$ by \cite[I.6.9]{Jantzen:2003}. By hypothesis, $\Ext_{U_r}^1(k,L(\lambda))$ is semisimple for $B/U_r$. Then there exist weights $\nu_1,\ldots,\nu_n \in X(T)$ such that $\Ext_{U_r}^1(k,L(\lambda)) \cong \bigoplus_{i=1}^n \nu_n$ as a $B/U_r$-module. If $\Ext_{B_r}^1(k,L(\lambda) \otimes \mu) \neq 0$, then there exists $\nu \in \set{\nu_1,\ldots,\nu_n}$ such that $\nu + \mu \in p^rX(T)$. Relabeling the $\nu_i$ if necessary, we may assume that $\nu_i+\mu \in p^rX(T)$ for $1 \leq i \leq m$, where $m = \dim \Ext_{B_r}^1(k,L(\lambda) \otimes \mu)$, and $\nu_i+\mu \notin p^rX(T)$ for $m < i \leq n$. For $1 \leq i \leq m$, write $\nu_i + \mu = p^r \xi_i$ with $\xi_i \in X(T)$. Then $\Ext_{B_r}^1(k,L(\lambda) \otimes \mu) \cong \bigoplus_{i=1}^m p^r \xi_i$ as a $B/B_r$-module, and
\[
\textstyle \Ext_{G_r}^1(k,L(\lambda) \otimes H^0(\mu))^{(-r)} \cong \bigoplus_{i=1}^m (\ind_{B/B_r}^{G/G_r} p^r \xi_i)^{(-r)} \cong \bigoplus_{i=1}^m H^0(\xi_i). \qedhere
\]
\end{proof}

\subsection{Proof of Theorem \ref{theorem:vanishingconditionforExt1iso}} \label{subsection:proofofvanishingconditionforExt1iso}

We can now prove Theorem \ref{theorem:vanishingconditionforExt1iso}.

\begin{proof}[Proof of Theorem \ref{theorem:vanishingconditionforExt1iso}]
It suffices by Lemma \ref{lemma:ext1isotohom} to show for all $\mu \in X(T)_+$ with $\mu \neq 0$ that
\begin{equation} \label{eq:HomExt1vanishing}
\Hom_{G/G_r}(V(\mu)^{(r)},\Ext_{G_r}^1(k,L(\lambda) \otimes H^0(\mu))) = 0.
\end{equation}
So let $\mu \in X(T)_+$ with $\mu \neq 0$. By \cite[II.4.13]{Jantzen:2003}, and retaining the notation from the proof of Lemma \ref{lemma:goodfiltrationforExt1}, the $\Hom$-set in \eqref{eq:HomExt1vanishing} is nonzero only if $\mu = \xi_i$ for some $1 \leq i \leq m$. If $\mu = \xi_i$, then $\nu_i = p^r\xi_i - \mu = (p^r-1)\mu$ is a nonzero $\Tfq$-invariant weight in $\Ext_{U_r}^1(k,L(\lambda))$, a contradiction, because by assumption no such weights exist. Thus we conclude that \eqref{eq:HomExt1vanishing} holds.
\end{proof}

\subsection{The restriction map for second cohomology} \label{subsection:restrictionmapforH2}

So far we have studied the third and sixth terms in the long exact sequence \eqref{eq:lesrestriction} with $M = L(\lambda)$ and $\lambda \in X_r(T)$. Now we study the term $\Ext_G^2(k,M \otimes N)$, and conditions under which the restriction map $\opH^2(G,L(\lambda)) \rightarrow \opH^2(\Gfq,L(\lambda))$ is an isomorphism.

\begin{lemma} \label{lemma:Ext2toHomiso}
Let $\mu \in X(T)_+$, let $\lambda \in X_r(T)$, and set $M = L(\lambda)$. Suppose $\Ext_{U_r}^1(k,L(\lambda))$ is semisimple as a $B/U_r$-module. Then in \eqref{eq:LHSforvanishing}, $E_2^{i,1} = 0$ for all $i \geq 1$. In particular,
\[
\Ext_G^2(V(\mu)^{(r)}, L(\lambda) \otimes H^0(\mu)) \cong \Hom_{G/G_r}(V(\mu)^{(r)},\Ext_{G_r}^2(k,L(\lambda) \otimes H^0(\mu))).
\]
\end{lemma}

\begin{proof}
First, $E_2^{i,1} \cong \Ext_G^i(V(\mu),\Ext_{G_r}^1(k,L(\lambda) \otimes H^0(\mu))^{(-r)})=0$ for all $i \geq 1$ by Lemma \ref{lemma:goodfiltrationforExt1} and \cite[II.4.16]{Jantzen:2003}. Next, $E_2^{i,0} = 0$ for all $i \geq 1$ by Lemma \ref{lemma:E2i0vanishing}. It follows then that
\[
E_2^{0,2} \cong E_\infty^{0,2} \cong \Ext_G^2(V(\mu)^{(r)}, L(\lambda) \otimes H^0(\mu)),
\]
whence the isomorphism of the lemma.
\end{proof}

\begin{lemma} \label{lemma:HomtoExt2vanish}
Let $\mu \in X(T)_+$ with $\mu \neq 0$, and let $\lambda \in X_r(T)$. Suppose that $\Ext_{U_r}^2(k,L(\lambda))^{\Tfq} = \Ext_{U_r}^2(k,L(\lambda))^T$, and that $p^r > \max \{-(\nu,\gamma^\vee): \gamma \in \Delta, \nu \in X(T),\Ext_{U_r}^1(k,L(\lambda))_\nu \neq 0 \}$. Then
\begin{equation} \label{eq:HomtoExt2vanish}
\Hom_{G/G_r}(V(\mu)^{(r)},\Ext_{G_r}^2(k,L(\lambda) \otimes H^0(\mu))) = 0.
\end{equation}
\end{lemma}

\begin{proof}
By \cite[II.12.2]{Jantzen:2003}, there exists a spectral sequence of $G$-modules with
\begin{equation} \label{eq:inversionspecseq}
E_2^{i,j} = R^i \ind_{B/B_r}^{G/G_r} \Ext_{B_r}^j(k,L(\lambda) \otimes \mu) \Rightarrow \Ext_{G_r}^{i+j}(k,L(\lambda) \otimes H^0(\mu)).
\end{equation}
We claim for all $i \geq 1$ that $E_2^{i,0} = E_2^{i,1} = 0$. Indeed, write $\mu = \mu_0+p^r\mu_1$ with $\mu_0 \in X_r(T)$ and $\mu_1 \in X(T)_+$. Then as argued in the proof of Lemma \ref{lemma:E2i0vanishing}, $\Hom_{B_r}(k,L(\lambda) \otimes \mu)$ is either the zero module or a one-dimensional $B/B_r$-module of weight $p^r \mu_1$. If the former, then $E_2^{i,0}=0$ for all $i\geq 0$, and if the latter, then $E_2^{i,0}=0$ for all $i \geq 1$ by Kempf's vanishing theorem. Next suppose $\Ext_{B_r}^1(k,L(\lambda) \otimes \mu) \neq 0$. As in the proof of Lemma \ref{lemma:goodfiltrationforExt1}, the weights of $\Ext_{B_r}^1(k,L(\lambda) \otimes \mu)$ are elements of $p^r X(T)$ of the form $\nu + \mu$ for $\nu$ a weight of $\Ext_{U_r}^1(k,L(\lambda))$. So let $\nu$ be a weight of $\Ext_{U_r}^1(k,L(\lambda))$, and suppose $\nu + \mu = p^r \xi$ for some $\xi \in X(T)$. Then for all $\gamma \in \Delta$,
\[
p^r (\xi,\gamma^\vee)-(\nu,\gamma^\vee) = (\mu,\gamma^\vee) \geq 0.
\]
By the assumption on $p^r$, this implies for all $\gamma \in \Delta$ that $(\xi,\gamma^\vee) \geq 0$. Then every weight of $\Ext_{B_r}^1(k,L(\lambda) \otimes \mu)$ is dominant, so it follows from Kempf's vanishing theorem that $E_2^{i,1} = 0$ for all $i \geq 1$. This finishes the verification of the claim.

Since $E_2^{i,0} = E_2^{i,1} = 0$ for all $i \geq 1$, we obtain from \eqref{eq:inversionspecseq} the $G$-module isomorphism
\begin{equation} \label{eq:GrExtasinduction}
\Ext_{G_r}^2(k,L(\lambda) \otimes H^0(\mu)) \cong \ind_{B/B_r}^{G/G_r} \Ext_{B_r}^2(k,L(\lambda) \otimes \mu),
\end{equation}
and hence by Frobenius Reciprocity the isomorphism
\[
\Hom_{G/G_r}(V(\mu)^{(r)},\Ext_{G_r}^2(k,L(\lambda) \otimes H^0(\mu))) \cong \Hom_{B/B_r}(V(\mu)^{(r)},\Ext_{B_r}^2(k,L(\lambda) \otimes \mu)).
\]
If the latter space is nonzero, then $p^r\mu$ must be a weight of $\Ext_{B_r}^2(k,L(\lambda) \otimes \mu)$, because $V(\mu)$ is generated as a $B$-module by its $\mu$-weight space. So suppose $p^r \mu$ is a weight of $\Ext_{B_r}^2(k,L(\lambda) \otimes \mu)$. Then as in the previous paragraph, there exists a weight $\nu$ of $\Ext_{U_r}^2(k,L(\lambda))$ such that $\nu + \mu = p^r \mu$, and hence $\nu = (p^r-1)\mu$ is a nonzero $\Tfq$-invariant weight in $\Ext_{U_r}^2(k,L(\lambda))$, a contradiction. Thus, the $\Hom$-set is zero and \eqref{eq:HomtoExt2vanish} holds.
\end{proof}

\begin{remark} \label{remark:boundonExt2dimension}
Retain the assumptions on $\mu$, $\lambda$, and $p^r$ from Lemma \ref{lemma:HomtoExt2vanish}. Then the proof of the lemma shows that each weight of $\Ext_{B_r}^1(k,L(\lambda) \otimes \mu)$ is dominant. Next suppose in addition that $\Ext_{U_r}^1(k,L(\lambda))$ is semisimple as a $B/U_r$-module. Then Lemma \ref{lemma:Ext2toHomiso} and the proof of Lemma \ref{lemma:HomtoExt2vanish} show that
\begin{align*}
\dim \Ext_G^2(V(\mu)^{(r)},L(\lambda) \otimes H^0(\mu)) &= \dim \Hom_{B/B_r}(V(\mu)^{(r)},\Ext_{B_r}^2(k,L(\lambda) \otimes \mu)) \\
&\leq \dim \Ext_{U_r}^2(k,L(\lambda))_{(p^r-1)\mu}.
\end{align*}
\end{remark}

\subsection{Proof of Theorem \ref{theorem:Ext2Gfqiso}} \label{subsection:proofofExt2Gfqiso}

We now prove Theorem \ref{theorem:Ext2Gfqiso}.

\begin{proof}[Proof of Theorem \ref{theorem:Ext2Gfqiso}]
By Theorem  \ref{theorem:vanishingconditionforExt1iso} and Lemmas \ref{lemma:Ext2toHomiso} and \ref{lemma:HomtoExt2vanish}, the vanishing condition \eqref{eq:vanishingcondition} is satisfied for $M=L(\lambda)$ and $i \in \set{1,2}$. Then the terms $\Ext_G^1(k,L(\lambda) \otimes N)$ and $\Ext_G^2(k,L(\lambda) \otimes N)$ in \eqref{eq:lesrestriction} are both zero, so the restriction maps $\opH^2(G,L(\lambda)) \rightarrow \opH^2(\Gfq,L(\lambda))$ and $\opH^3(G,L(\lambda)) \rightarrow \opH^3(\Gfq,L(\lambda))$ are an isomorphism and an injection, respectively.
\end{proof}

\begin{remark}
If we drop the assumption $\Ext_{U_r}^1(k,L(\lambda))^{\Tfq} = \Ext_{U_r}^1(k,L(\lambda))^T$ in Theorem \ref{theorem:Ext2Gfqiso}, then we still get that the restriction maps 
$\opH^2(G,L(\lambda)) \rightarrow \opH^2(\Gfq,L(\lambda))$ and $\opH^3(G,L(\lambda)) \rightarrow \opH^3(\Gfq,L(\lambda))$ are surjective and injective, respectively.
\end{remark}

\subsection{An alternate description of second cohomology} \label{subsection:alternatecalculation}

Even in cases where the hypotheses of Theorem \ref{theorem:Ext2Gfqiso} do not hold, we can, under suitable conditions, still describe the second cohomology group $\opH^2(\Gfq,L(\lambda))$ in terms of rational cohomology for the full algebraic group $G$. In particular, this approach enables us in Section \ref{subsection:additionalcalculations} to compute the cohomology group $\opH^2(\Gfq,L(\lambda))$ in cases where it is not isomorphic to $\opH^2(G,L(\lambda))$. For this approach we utilize constructions and techniques developed in \cite[\S\S 2.7--2.8]{Bendel:2011a}. Specifically, given $\sigma\in X(T)_{+}$, there exist submodules $S_{< \sigma}$ (resp.\ $S_{\leq \sigma}$) and quotients  $Q_{\nless \sigma}$ (resp.\ $Q_{\nleq \sigma}$) of ${\mathcal G}_{r}(k)$ with the following properties: 

\begin{enumroman}
\item $S_{< \sigma}$  (resp.\ $S_{\leq \sigma}$) has a filtration with factors of the form $H^0(\nu)\otimes H^0(\nu^*)^{(r)}$, where $\nu < \sigma$ (resp.\ $\nu \leq \sigma$) and $\nu $ is linked to $\sigma$.
\item $Q_{\nless \sigma}$  (resp.\ $Q_{\nleq \sigma}$) has a filtration with factors of the form $H^0(\nu)\otimes H^0(\nu^*)^{(r)}$, where $\nu \nless \sigma$ (resp.\ $\nu \nleq \sigma$) or $\nu $ is not linked to $\sigma$.
\item Each such factor occurs in the aforementioned filtrations with multiplicity one.
\item There exists a short exact sequence of $G$-modules
\[
0 \to H^{0}(\sigma)\otimes H^{0}(\sigma^{*})^{(r)} \rightarrow Q_{\nless \sigma} \to Q_{\nleq\sigma} \to 0.
\] 
\item There exists a short exact sequence of $G$-modules
\[
0 \to S_{< \sigma} \to {\mathcal G}_{r}(k) \to Q_{\nless \sigma} \to 0.
\] 
\item There exists a short exact sequence of $G$-modules
\[
0 \to S_{\leq \sigma} \to {\mathcal G}_{r}(k)  \to Q_{\nleq \sigma} \to 0.
\]
\end{enumroman}

With these submodules and quotients we prove the following theorem:
 
\begin{theorem} \label{theorem:alternatecalculation}
Let $\lambda \in X(T)_+$ satisfy the following properties: 
\begin{enumalph} 
\item $L(\lambda)=H^{0}(\lambda)$;
\item There exists $\sigma\in X(T)_{+}$ such that for all $\mu \neq \sigma$, $\Ext^{2}_{G}(k,H^{0}(\mu)\otimes H^{0}(\mu^{*})^{(r)}\otimes L(\lambda)) = 0$;
\item $\Ext^{3}_{G}(k,S_{<\sigma}\otimes L(\lambda))=0$; and
\item $\Ext^{1}_{G}(k,Q_{\nleq \sigma}\otimes L(\lambda))=0$.
\end{enumalph}
Then $\opH^{2}(\Gfq,L(\lambda))\cong \Ext_G^2(k,H^{0}(\sigma)\otimes H^{0}(\sigma^{*})^{(r)}\otimes L(\lambda))$. 
\end{theorem} 

\begin{proof}
The assumption $L(\lambda) = H^0(\lambda)$ together with \cite[II.4.13]{Jantzen:2003} and the long exact sequence \eqref{eq:lesrestriction} yield the isomorphism $\opH^2(\Gfq,L(\lambda)) \cong \Ext_G^2(k,L(\lambda) \otimes N)$. The latter space is seen to be isomorphic to $\Ext_G^2(k,L(\lambda) \otimes \cgr(k))$ by considering the long exact sequence in cohomology arising from the short exact sequence \eqref{eq:sesM} with $M = L(\lambda)$. Next, consider the long exact sequence in cohomology associated to the following short exact sequence arising from (v) above:
\[
0 \rightarrow S_{< \sigma}\otimes L(\lambda) \rightarrow {\mathcal G}_{r}(k)\otimes L(\lambda) \rightarrow Q_{\nless \sigma}\otimes L(\lambda) \rightarrow 0.
\]
Assumptions (b) and (c), together with the description in (i) above for the filtration on $S_{<\sigma}$, imply that $\Ext_G^2(k,\cgr(k) \otimes L(\lambda)) \cong \Ext_G^2(k,Q_{\nless \sigma} \otimes L(\lambda))$. Finally, consider the long exact sequence in cohomology associated to the following short exact sequence arising from (iv) above:
\[
0 \rightarrow H^{0}(\sigma)\otimes H^{0}(\sigma^{*})^{(r)} \otimes L(\lambda)\rightarrow Q_{\nless \sigma} \otimes L(\lambda) \rightarrow Q_{\nleq \sigma} \otimes L(\lambda) \rightarrow 0.
\]
Then assumptions (b) and (d) imply the isomorphism
\[
\Ext_G^2(k,Q_{\nless \sigma} \otimes L(\lambda)) \cong \Ext_G^2(k,H^0(\sigma) \otimes H^0(\sigma^*)^{(r)} \otimes L(\lambda)).
\]
Now the conclusion of the theorem follows by applying this sequence of isomorphisms. 
\end{proof}

\section{First cohomology for the Frobenius kernel \texorpdfstring{$U_r$}{Ur}} \label{section:Urcohomology}

In this section we generalize the results of \cite[\S 3]{UGA:2011} to the higher Frobenius kernels $U_r$ of $U$.

\subsection{The socle of \texorpdfstring{$U_r$}{Ur} cohomology}

Our first step is to analyze for $\lambda \in X_r(T)$ the socle of the rational $B/U_r$-module $\Ext_{U_r}^1(L(\lambda),k)$. Recall that the irreducible rational $B/U_r$-modules are one-dimensional, hence are determined by characters in $X(B/U_r) = X(T)$. Given $\sigma \in X(T)$, the dimension of the $(-\sigma)$-isotypic component in the socle of a rational $B/U_r$-module $M$ is equal to $\dim \Hom_{B/U_r}(-\sigma,M)$. Since $\sigma$ can be written uniquely in the form $\sigma = \mu + p^r\nu$ with $\mu \in X_r(T)$ and $\nu \in X(T)$, to determine the socle of $\Ext_{U_r}^1(L(\lambda),k)$ as a $B/U_r$-module it suffices as in \cite[\S 3.1]{UGA:2011} to determine the dimensions of the $\Hom$-spaces
\begin{equation} \label{eq:HomforUrsocle}
\Hom_{B/U_r}(-\mu-p^r\nu,\Ext_{U_r}^1(L(\lambda),k)) \cong \Hom_{B/B_r}(k,\Ext_{B_r}^1(L(\lambda),\mu+p^r\nu)).
\end{equation}

\begin{proposition} \label{proposition:HomforUrsocle}
Suppose $p > 2$. Let $\lambda,\mu \in X_r(T)$ and $\nu \in X(T)$. Then
\[
\Hom_{B/U_r}(-\mu-p^r\nu,\Ext_{U_r}^1(L(\lambda),k)) \cong \begin{cases} \Ext_B^1(L(\lambda),\mu+p^r\nu) & \text{if $\lambda \neq \mu$,} \\ 0 & \text{if $\lambda = \mu$.} \end{cases}
\]
\end{proposition}

\begin{proof}
If $\lambda = \mu$, then $\Ext_{B_r}^1(L(\lambda),\lambda+p^r\nu) \cong \Ext_{B_r}^1(L(\lambda),\lambda) \otimes p^r\nu = 0$ by \cite[II.12.6]{Jantzen:2003}, and consequently $\Hom_{B/U_r}(-\mu-p^r\nu,\Ext_{U_r}^1(L(\lambda),k)) = 0$ by \eqref{eq:HomforUrsocle}. So assume $\lambda \neq \mu$, and consider the LHS spectral sequence
\[
E_2^{i,j} = \Ext_{B/B_r}^i(k,\Ext_{B_r}^j(L(\lambda),\mu+p^r\nu)) \Rightarrow \Ext_B^{i+j}(L(\lambda),\mu+p^r\nu).
\]
It gives rise to the five-term exact sequence
\begin{equation} \label{eq:5termforUrsocle}
0 \rightarrow E_2^{1,0} \rightarrow \Ext_B^1(L(\lambda),\mu+p^r\nu) \rightarrow E_2^{0,1} \rightarrow E_2^{2,0} \rightarrow \Ext_B^2(L(\lambda),\mu+p^r\nu).
\end{equation}
One has $\Hom_{B_r}(L(\lambda),\mu+p^r\nu) = \Hom_{B_r}(L(\lambda),\mu) \otimes p^r\nu = 0$ because $\lambda,\mu \in X_r(T)$, $\lambda \neq \mu$, and $L(\lambda)$ is generated as $B_r$-module by its highest weight space (cf.\ \cite[II.3.14]{Jantzen:2003}). Then $E_2^{1,0} = E_2^{2,0} = 0$, so $\Hom_{B/U_r}(-\mu-p^r\nu,\Ext_{U_r}^1(L(\lambda),k)) = E_2^{0,1} \cong \Ext_B^1(L(\lambda),\mu+p^r\nu)$.
\end{proof}

\begin{corollary} \label{corollary:restrictionBtoUr}
Suppose $p > 2$ and let $\lambda \in X_r(T)$. Then the restriction map
\[
\Ext_B^1(L(\lambda),k) \rightarrow \Ext_{U_r}^1(L(\lambda),k)^{\Tfq}
\]
is an injection.
\end{corollary}

\begin{proof}
If $\lambda = 0$, then $\Ext_B^1(k,k) \cong \Ext_G^1(k,k) = 0$ by \cite[II.4.11]{Jantzen:2003}. On the other hand, if $\lambda \neq 0$, then taking $\mu = \nu = 0$ in the five-term exact sequence \eqref{eq:5termforUrsocle}, and applying \eqref{eq:HomforUrsocle}, we get
\begin{equation} \label{eq:BtoUrTfqinvariants}
\begin{split}
\Ext_B^1(L(\lambda),k) &\cong \Ext_{B_r}^1(L(\lambda),k)^{B/B_r} \\
&\cong \Ext_{U_r}^1(L(\lambda),k)^{B/U_r} \\
& \subseteq \Ext_{U_r}^1(L(\lambda),k)^T \\
&\subseteq \Ext_{U_r}^1(L(\lambda),k)^{\Tfq},
\end{split}
\end{equation}
in which the isomorphisms are induced by the corresponding restriction maps.
\end{proof}

Given $\lambda,\mu \in X(T)$, write $\mu \uparrow \lambda$ for the order relation on $X(T)$ defined in \cite[II.6.4]{Jantzen:2003}.

\begin{theorem} \label{theorem:socleofUrcohomology}
Suppose $p>2$ and let $\lambda \in X_r(T)$. Then
\[
\soc_{B/U_r} \Ext_{U_r}^1(L(\lambda),k) \cong \hspace{-1em} \bigoplus_{\substack{\alpha \in \Delta \\ (\lambda,\alpha^\vee) = ap^n-1 \\ (\lambda,\alpha^\vee) \neq p^r-1 \\ 0 \leq n < r \\ 1 \leq a \leq p}} \hspace{-1em} -s_\alpha \cdot \lambda
\oplus \hspace{-2em} \bigoplus_{\substack{\alpha \in \Delta \\ (a-1)p^n \leq (\lambda,\alpha^\vee) < ap^n-1 \\ 0 < n < r \\ 1 \leq a < p}} \hspace{-2em} -(\lambda-ap^n\alpha)
\oplus \bigoplus_{\substack{\sigma \in X(T)_+ \\ \sigma < \lambda}} (-\sigma)^{\oplus m_\sigma}
\]
where $m_\sigma = \dim \Ext_G^1(L(\lambda),H^0(\sigma))$.
\end{theorem}

\begin{proof}
The strategy is the same as that for the proof of \cite[Theorem 3.2.1]{UGA:2011}. Let $\mu \in X_r(T)$ with $\mu \neq \lambda$, and let $\nu \in X(T)$. Then by Proposition \ref{proposition:HomforUrsocle}, $(-\mu - p^r\nu)$ occurs with multiplicity $\dim \Ext_B^1(L(\lambda),\mu + p^r \nu)$ in $\soc_{B/U_r} \Ext_{U_r}^1(L(\lambda),k)$. First suppose $\mu + p^r \nu \in X(T)_+$. Then the same argument as in the proof of \cite[Theorem 3.2.1]{UGA:2011} (replacing $p$ by $p^r$ and $X_1(T)$ by $X_r(T)$) shows that $(-\mu - p^r \nu)$ occurs in $\soc_{B/U_r} \Ext_{U_r}^1(L(\lambda),k)$ with multiplicity $m_{\mu + p^r \nu}$, and conversely that every dominant weight $\sigma \in X(T)_+$ with $m_\sigma \neq 0$ satisfies $\sigma < \lambda$, hence can be written in the form $\mu + p^r \nu$ with $\mu \in X_r(T)$, $\nu \in X(T)$, and $\mu \neq \lambda$.

Now suppose $\mu + p^r \nu \notin X(T)_+$. Then by \cite[Proposition 2.3]{Andersen:1984a}, $\Ext_B^1(L(\lambda),\mu + p^r \nu) \neq 0$ only if one of the following mutually exclusive conditions is satisfied, in which case $\Ext_B^1(L(\lambda),\mu + p^r \nu) \cong k$:
\begin{enumerate}[leftmargin=*]
\item $\mu + p^r \nu = s_\alpha \cdot \lambda$ for some $\alpha \in \Delta$ with $(\lambda,\alpha^\vee) = ap^n - 1$, $n \geq 0$ and $1 \leq a \leq p$; or
\item $\mu + p^r \nu = \lambda - ap^n \alpha$ for some $\alpha \in \Delta$ with $(a-1)p^n \leq (\lambda,\alpha^\vee) < ap^n - 1$, $n > 0$, and $1 \leq a < p$.
\end{enumerate}
Since $\lambda \in X_r(T)$, (1) is satisfied only if $ap^n \leq p^r$, or equivalently, only if $0 \leq n < r$ and $1 \leq a \leq p$. And since $\mu \neq \lambda$, (1) is satisfied only if $(\lambda,\alpha^\vee) \neq p^r-1$, and (2) is satisfied only if $n < r$.
\end{proof}

\begin{corollary} \label{corollary:socleofUrcohomology}
Suppose $\lambda \in X(T)_+$ is a dominant root or is less than or equal to a fundamental weight. Assume that $p > 3$. Then
\[
\soc_{B/U_r} \Ext_{U_r}^1(L(\lambda),k) \cong
\bigoplus_{\alpha \in \Delta} -s_\alpha \cdot \lambda
\oplus \bigoplus_{\substack{\alpha \in \Delta \\ 0 < n < r}} -(\lambda - p^n \alpha)
\oplus \bigoplus_{\substack{\sigma \in X(T)_+ \\ \sigma < \lambda}} (-\sigma)^{\oplus m_\sigma}.
\]
\end{corollary}

\begin{proof}
From the list of possible values for $\lambda$ (cf.\ Table \ref{table:problemweights} and \cite[\S 7.1]{UGA:2011}), one can verify for all $\alpha \in \Delta$ that $0 \leq (\lambda,\alpha^\vee) \leq 3$. Then the result follows from Theorem \ref{theorem:socleofUrcohomology}.
\end{proof}

\begin{remark} \label{remark:minimaldomroot}
The only dominant weights less than or equal to the highest root $\wtalpha$ are $\wtalpha$, the highest short root $\alpha_0$, and the zero weight. By direct inspection, one observes that if $\wtalpha \neq \alpha_0$, then $\alpha_0$ is a fundamental weight. Thus, the dominant weights $\sigma$ that occur in Corollary  \ref{corollary:socleofUrcohomology} are all less than or equal to a fundamental weight.
\end{remark}

\subsection{Semisimplicity of \texorpdfstring{$U_r$}{Ur} cohomology}

Set $M = \Ext_{U_r}^1(L(\lambda),k)$. Our goal now is to show under certain restrictions on $p$ and $\lambda$ that the socle of $M$ is equal to the entire module. Given $\mu \in X(T)$, let $I_r(\mu)$ be the injective hull of $\mu$ in the category of rational $B/U_r$-modules. Then $I_r(\mu) \cong k[U/U_r] \otimes \mu$ as a $B/U_r$-module, where $k[U/U_r]$ is the coordinate ring of the unipotent group $U/U_r$ (cf.\ \cite[I.3.11]{Jantzen:2003}).

Suppose $\lambda \in X(T)_+$ is a dominant root or is less than or equal to a fundamental weight, and assume that $p > 3$. Set
\begin{equation} \label{eq:Q}
Q = \bigoplus_{\alpha \in \Delta} I_r(-s_\alpha \cdot \lambda)
\oplus \bigoplus_{\substack{\alpha \in \Delta \\ 0 < n < r}} I_r(-\lambda + p^n \alpha)
\oplus \bigoplus_{\substack{\sigma \in X(T)_+ \\ \sigma < \lambda}} I_r(-\sigma)^{\oplus m_\sigma}.
\end{equation}
Then $\soc_{B/U_r} M$ is a submodule of $\soc_{B/U_r} Q$ by Corollary \ref{corollary:socleofUrcohomology}, so there exists an injection of $B/U_r$-modules $M \hookrightarrow Q$. To show that $\soc_{B/U_r} M = M$, it suffices to show that no weight from the second socle layer of $Q$ can be a weight of $M$.

\begin{lemma} \label{lemma:secondsoclelayer}
Let $\mu \in X(T)$. Then the second socle layer of the $B/U_r$-module $I_r(\mu)$ consists of one-dimensional modules of the form $\mu + p^m \gamma$ with $\gamma \in \Delta$ and $m \geq r$.
\end{lemma}

\begin{proof}
The argument establishing that the weights in the second socle layer of $I_r(\mu)$ have the form $\mu + p^m \gamma$ is similar to the argument in the proof of \cite[Lemma 3.3.2]{UGA:2011}, replacing $U_1$ and $I(\mu)$ there by $U_r$ and $I_r(\mu)$. To conclude that $m \geq r$ and not merely $m > 0$, we use the fact that the weights of $T$ in $k[U/U_r]$ are all elements of $p^r X(T)$.
\end{proof}

\begin{lemma} \label{lemma:weightsofH1Ur}
Let $V$ be a finite-dimensional rational $B$-module. Let $\mu$ be a weight of $T$ in $\Ext_{U_r}^1(V,k) \cong \opH^1(U_r,V^*)$. Then $\mu = p^i \beta + \nu$ for some $\beta \in \Delta$, some integer $0 \leq i < r$, and some weight $\nu$ of $V^* = \Hom_k(V,k)$.
\end{lemma}

\begin{proof}
First, if $\mu$ is a weight of $T$ in $\opH^1(U_r,V^*)$, then $\mu$ is also a weight of $T$ in $\opH^1(U_r,k) \otimes V^*$ by the argument in \cite[\S 2.5]{UGA:2009}, so it suffices to describe the weights of $\opH^1(U_r,k)$. This we do by induction on $r$. First, if $r=1$, then $\opH^1(U_1,k)$ is by \cite[I.9.20]{Jantzen:2003} a $T$-module subquotient of $\opH^1(\fu,k)$, the ordinary Lie algebra cohomology of $\fu = \Lie(U)$. Since $\opH^1(\fu,k)$ is a $T$-module subquotient of $\Hom_k(\fu/[\fu,\fu],k)$, as can be seen by inspecting the Koszul resolution for computing Lie algebra cohomology, we conclude that the weights of $T$ in $\opH^1(U_1,k)$ must all be simple roots. For the induction step, consider the LHS spectral sequence $E_2^{i,j} = \opH^i(U_r/U_{r-1},\opH^j(U_{r-1},k)) \Rightarrow \opH^{i+j}(U_r,k)$. Then $\opH^1(U_r,k)$ is a $T$-module subquotient of $E_2^{1,0} \oplus E_2^{0,1} \cong \opH^1(U_r/U_{r-1},k) \oplus \opH^1(U_{r-1},k)^{U_r}$. By induction, the weights of $T$ in $\opH^1(U_{r-1},k)$ have the form $p^i \beta$ with $\beta \in \Delta$ and $0 \leq i < (r-1)$, while $\opH^1(U_r/U_{r-1},k)$ is isomorphic as a $T$-module to $\opH^1(U_1,k)^{(r-1)}$, which has weights of the form $p^i \beta$ with $\beta \in \Delta$ and $0 \leq i < r$.
\end{proof}

\begin{remark} \label{remark:twistedschemes}
Let $r \geq s \geq 0$. Given a commutative $k$-algebra $A$, define $A^{(s)}$ to be the $k$-algebra that coincides with $A$ as a ring, but with the $k$-module structure modified so that $b \in k$ acts on $A^{(s)}$ as $b^{p^{-s}}$ acts on $A$. Then the Frobenius morphism induces an isomorphism $U_r/U_s \cong (U^{(s)})_{r-s}$, where $U^{(s)}$ is the $k$-group scheme with coordinate ring $k[U]^{(s)}$. Since $U$ admits an $\F_p$-structure, $k[U]^{(s)}$ identifies with $k[U]$ via the arithmetic Frobenius endomorphism, and this induces identifications of $U^{(s)}$ with $U$ and $(U^{(s)})_{r-s}$ with $U_{r-s}$. Under these identifications, the ordinary conjugation action of $B$ on $U_r/U_s$ corresponds to the conjugation action of $B$ on $U_{r-s}$ twisted by $F^s$. It then follows that there exists a graded $B$-module isomorphism $\Hbul(U_r/U_s,k) \cong \Hbul(U_{r-s},k)^{(s)}$. For further details, consult \cite[I.9.2, I.9.5]{Jantzen:2003}.
\end{remark}

We now give conditions under which $\Ext_{U_r}^1(L(\lambda),k)$ is semisimple as a $B/U_r$-module.

\begin{theorem} \label{theorem:semisimplicity}
Suppose $\lambda \in X(T)_+$ is a dominant root or is less than or equal to a fundamental weight. Assume that $p > 5$ if $\Phi$ is of type $E_8$ or $G_2$, and $p > 3$ otherwise. Then as a $B/U_r$-module, $\Ext_{U_r}^1(L(\lambda),k) = \soc_{B/U_r} \Ext_{U_r}^1(L(\lambda),k)$, that is,
\[
\Ext_{U_r}^1(L(\lambda),k) \cong
\bigoplus_{\alpha \in \Delta} -s_\alpha \cdot \lambda
\oplus \bigoplus_{\substack{\alpha \in \Delta \\ 0 < n < r}} -(\lambda - p^n \alpha)
\oplus \bigoplus_{\substack{\sigma \in X(T)_+ \\ \sigma < \lambda}} (-\sigma)^{\oplus m_\sigma}
\]
where $m_\sigma = \dim \Ext_G^1(L(\lambda),H^0(\sigma))$.
\end{theorem}

\begin{proof}
The strategy is the same as for the proof of \cite[Theorem 3.4.1]{UGA:2011}. Set $M = \Ext_{U_r}^1(L(\lambda),k)$, and let $Q$ be the injective $B/U_r$-module defined in \eqref{eq:Q}. Define submodules $Q_1,Q_2,Q_3$ of $Q$ by
\begin{align*} 
Q_1 &= \textstyle \bigoplus_{\alpha \in \Delta} I_r(-s_\alpha \cdot \lambda), \\
Q_2 &= \textstyle \bigoplus_{\alpha \in \Delta, 0 < n < r} I_r(-\lambda + p^n\alpha), \text{ and} \\
Q_3 &= \textstyle \bigoplus_{\sigma \in X(T)_+,\sigma < \lambda} I_r(-\sigma)^{\oplus m_\sigma}.
\end{align*}
The goal is to show that no weight from the second socle layer of $Q$ is a weight of $M$.

First suppose that a weight from the second socle layer of $Q_1$ is a weight of $M$. Then by Lemmas \ref{lemma:secondsoclelayer} and \ref{lemma:weightsofH1Ur}, there exist simple roots $\alpha,\beta,\gamma \in \Delta$, integers $m \geq r$ and $0 \leq i < r$, and a weight $\nu$ of $L(\lambda)$ such that $-s_\alpha \cdot \lambda + p^m \gamma = p^i \beta - \nu$, that is, such that
\begin{equation} \label{eq:weightequationQ1}
\lambda - \nu = -p^i\beta + p^m \gamma + (\lambda + \rho,\alpha^\vee)\alpha.
\end{equation}
Since $\lambda \geq \nu$ and $\alpha$, $\beta$, and $\gamma$ are simple, this implies that $\beta \in J:=\set{\alpha,\gamma}$, and hence that $\nu$ is a weight of the $P_J$-module $H_J^0(\lambda)$; see \cite[II.5.21]{Jantzen:2003}. If $(\Phi,J,\lambda) = (B_2,\Delta,\wtalpha)$, then one can verify by hand for every weight $\nu$ of $H^0(\wtalpha)$ that the difference $\wtalpha - \nu$ cannot be written the form \eqref{eq:weightequationQ1}. So suppose $(\Phi,J,\lambda) \neq (B_2,\Delta,\wtalpha)$. Observe that the expression on the right-hand side of \eqref{eq:weightequationQ1} has height $-p^i + p^m + (\lambda + \rho,\alpha^\vee) \geq p$. In \cite[\S 7.2]{UGA:2011} we computed all possible differences $\lambda - \nu$ for $\nu$ a weight of $H_J^0(\lambda)$ under the assumptions that $\lambda$ is not orthogonal to $\Phi_J$ and $(\Phi,J,\lambda) \neq (B_2,\Delta,\wtalpha)$. From the calculations there, one sees by the assumptions on $p$ that the difference $\lambda - \nu$ always has height strictly less than $p$, except for a few cases when $(\Phi,J,\lambda,p) = (G_2,\Delta,\omega_2,7)$; for these extra cases in type $G_2$, one can rule out a solution to \eqref{eq:weightequationQ1} by hand. Thus, if $\lambda$ is not orthogonal to $\Phi_J$ and if $(\Phi,J,\lambda) \neq (B_2,\Delta,\wtalpha)$, then \eqref{eq:weightequationQ1} has no solution. Finally, if $\lambda$ is orthogonal to $\Phi_J$, that is, if $(\lambda,\alpha^\vee) = (\lambda,\gamma^\vee) = 0$, then $\dim H_J^0(\lambda) = 1$ and $\lambda - \nu = 0$. This also contradicts the lower bound on the height of $\lambda - \nu$, so we conclude in all cases that no weight from the second socle layer of $Q_1$ is a weight of $M$.

Next suppose that a weight from the second socle layer of $Q_2$ is a weight of $M$. Then as in the previous paragraph, there exist simple roots $\alpha,\beta,\gamma \in \Delta$, integers $0 < n < r$, $m \geq r$, and $0 \leq i < r$, and a weight $\nu$ of $L(\lambda)$ such that $-\lambda + p^n \alpha + p^m \gamma = p^i \beta - \nu$, that is, such that
\begin{equation} \label{eq:weightequationQ2}
\lambda - \nu = -p^i \beta + p^m \gamma + p^n \alpha.
\end{equation}
As in the previous paragraph, $\nu$ must be a weight of $H_J^0(\lambda)$, where $J = \set{\alpha,\gamma}$. The right-hand side of \eqref{eq:weightequationQ2} has height $-p^i + p^m + p^n \geq 2p-1$. Then an analysis similar to that of the previous paragraph shows for every weight $\nu$ of $H_J^0(\lambda)$ that the difference $\lambda - \nu$ has height strictly less than $2p-1$. Thus, no solution to \eqref{eq:weightequationQ2} is possible, and we conclude that no weight from the second socle layer of $Q_2$ is a weight of $M$.

Finally, suppose that a weight from the second socle layer of $Q_3$ is a weight of $M$. Then there exists a dominant weight $\sigma \in X(T)_+$ with $\sigma < \lambda$, simple roots $\beta,\gamma \in \Delta$, integers $m \geq r$ and $0 \leq i < r$, and a weight $\nu$ of $L(\lambda)$ such that $-\sigma + p^m \gamma = p^i \beta - \nu$. Then
\[
(-\nu,\gamma^\vee) + (\sigma,\gamma^\vee) = 2p^m - p^i(\beta,\gamma^\vee) \geq 2p^m - 2p^i \geq 2(p-1).
\]
Suppose for the moment that $\Phi$ is not of type $E_8$. Then $(-\nu,\gamma^\vee) + (\sigma,\gamma^\vee) \geq 8$, while the list of possible values for $\sigma$ shows that $(\sigma,\gamma^\vee) \in \set{0,1,2}$, hence that $(-\nu,\gamma^\vee) \geq 6$. Choose $w \in W$ such that $\ol{\gamma}:= w\gamma$ is dominant, i.e., $\ol{\gamma} \in \set{\alpha_0,\wtalpha}$. Then $w\nu$ is a weight of $L(\lambda)$, so $w_0\lambda \leq w \nu \leq \lambda$, and we get $(-\nu,\gamma^\vee) = (-w\nu,\ol{\gamma}^\vee) \leq (-w_0\lambda,\ol{\gamma}^\vee)$. Observe that $\lambda^*:=-w_0\lambda$ is either a fundamental root or is a dominant weight less than or equal to a fundamental weight, so either $-w_0 \lambda \leq \wtalpha$ or $-w_0\lambda \leq \omega_i$ for some $i$. In either case, this implies $(-w_0\lambda,\ol{\gamma}^\vee) \leq 4$ by the assumption that $\Phi$ is not of type $E_8$. Then $(-\nu,\gamma^\vee) < 6$, a contradiction. Similarly, if $\Phi$ is of type $E_8$ and $p > 5$, then one derives a contradiction from the inequalities $(-\nu,\gamma^\vee) + (\sigma,\gamma^\vee) \geq 12$, $(\sigma,\gamma^\vee) \leq 3$, and $(-w_0\lambda,\ol{\gamma}^\vee) \leq 6$. Thus, no weight from the second socle layer of $Q_3$ is a weight of $M$.
\end{proof}

\section{Applications} \label{section:applications}

\subsection{An isomorphism for first cohomology}

In this section we demonstrate how to obtain, using our new techniques, most of the cases of the main theorem of \cite{UGA:2011}. We begin with an analysis of the weights in $\Ext_{U_r}^1(k,L(\lambda))$.

\begin{lemma} \label{lemma:Ext1weights}
Suppose $\lambda \in X(T)_+$ is a dominant root or is less than or equal to a fundamental weight. Assume $p > 5$ if $\Phi$ is of type $E_8$ or $G_2$, and $p > 3$ otherwise. Then $\Ext_{U_r}^1(k,L(\lambda))^{\Tfq} = \Ext_{U_r}^1(k,L(\lambda))^T$, and
\[
p^r > \max \{-(\nu,\gamma^\vee): \gamma \in \Delta, \nu \in X(T),\Ext_{U_r}^1(k,L(\lambda))_\nu \neq 0 \}.
\]
\end{lemma}

\begin{proof}
Observe that $\lambda^* := -w_0\lambda$ is also a dominant root or is a dominant weight less than or equal to a fundamental weight. Then by Theorem \ref{theorem:semisimplicity}, there exists an isomorphism of $B/U_r$-modules
\begin{equation} \label{eq:Ursemisimplicity}
\Ext_{U_r}^1(k,L(\lambda)) \cong \Ext_{U_r}^1(L(\lambda^*),k) \cong
\bigoplus_{\alpha \in \Delta} -s_\alpha \cdot \lambda^*
\oplus \bigoplus_{\substack{\alpha \in \Delta \\ 0 < n < r}} -(\lambda^* - p^n \alpha)
\oplus \bigoplus_{\substack{\sigma \in X(T)_+ \\ \sigma < \lambda^*}} (-\sigma)^{\oplus m_\sigma}
\end{equation}
where $m_\sigma = \dim \Ext_G^1(L(\lambda^*),H^0(\sigma))$. Using Table \ref{table:problemweights} and \cite[\S 7.1]{UGA:2011} for the lists of possible values for $\lambda$, and using the Cartan matrix to rewrite a given simple root as a sum of fundamental weights, one can check for all $\alpha \in \Delta$ that the assumptions on $p$ imply that $s_\alpha \cdot \lambda^* \notin (p^r-1)X(T)$. Similarly, by direct inspection of the possible $\sigma \in X(T)_+$ with $\sigma < \lambda^*$, one sees that $\sigma \in (p^r-1)X(T)$ only if $\sigma = 0$. Then to verify the first claim of the lemma, it suffices now to show for all $\alpha \in \Delta$ and $0 < n < r$ that $\lambda^* - p^n\alpha \in (p^r-1)X(T)$ only if $\lambda^* - p^n \alpha = 0$. So let $\alpha \in \Delta$ and $0 < n < r$; in particular, assume $r \geq 2$. Then for $\gamma \in \Delta$ one has
\begin{align*}
\abs{(\lambda^* - p^n\alpha,\gamma^\vee)} &\leq (\lambda^*,\gamma^\vee) + p^n \abs{(\alpha,\gamma^\vee)} \\
&< (p-1) + p^{r-1} (p-1) & \text{by the assumptions on $p$,} \\
&\leq p^r-1 & \text{because $r \geq 2$.}
\end{align*}
Then $(\lambda^* - p^n\alpha,\gamma^\vee) \in (p^r-1)\Z$ only if $(\lambda^* - p^n\alpha,\gamma^\vee) = 0$. Since $\gamma \in \Delta$ was arbitrary, we conclude $\lambda^* - p^n\alpha \in (p^r-1)X(T)$ only if $\lambda^* - p^n\alpha = 0$. This finishes the verification of the first claim of the lemma.

Now let $\alpha,\gamma \in \Delta$. Using the lists of possible values for $\lambda$, one can verify the following inequalities:
\begin{align*}
(s_\alpha \cdot \lambda^*,\gamma^\vee) &\leq 2 \quad \text{if $\Phi \in \set{D_n,E_6}$,} \\
(s_\alpha \cdot \lambda^*,\gamma^\vee) &\leq 3 \quad \text{if $\Phi \in \set{A_n,E_7}$,} \\
(s_\alpha \cdot \lambda^*,\gamma^\vee) &\leq 4 \quad \text{if $\Phi \in \set{B_n,C_n,E_8,F_4}$, and} \\
(s_\alpha \cdot \lambda^*,\gamma^\vee) &\leq 6 \quad \text{if $\Phi=G_2$.}
\end{align*}
For example, suppose $\Phi = D_n$. Then $\lambda^* = 0$ or $\lambda^* = \omega_j$ for some $j$. Write $(s_\alpha \cdot \lambda^*,\gamma^\vee) = (\lambda^*,\gamma^\vee) - (\lambda^*+\rho,\alpha^\vee)(\alpha,\gamma^\vee)$. We have $(\lambda^*,\gamma^\vee) \in \set{0,1}$, $(\lambda^*,\alpha^\vee) \in \set{0,1}$, and $(\alpha,\gamma^\vee) \geq -1$. Also, $(\lambda^*,\gamma^\vee) = (\lambda^*,\alpha^\vee) = 1$ only if $\alpha = \gamma$, in which case $(\alpha,\gamma^\vee)=2$. Combining these observations, we get the inequality $(s_\alpha \cdot \lambda^*,\gamma^\vee) \leq 2$. A similar elementary analysis establishes the inequalities for the other Lie types. In all cases, $p^r > -(-s_\alpha \cdot \lambda^*,\gamma^\vee)$. Similarly, one gets $p^r > -(-\sigma,\gamma^\vee)$ for all $\sigma \in X(T)_+$ with $\sigma < \lambda^*$ by considering the list of all possible values for $\sigma$. Finally, if $\alpha \in \Delta$ and $0 < n < r$, then the inequality $p^r > -(-(\lambda - p^n\alpha),\gamma^\vee)$ was observed in the previous paragraph.
\end{proof}

\begin{theorem} \textup{(cf.\ \cite[Theorem 1.2.1]{UGA:2011})} \label{theorem:H1Gfqiso}
Suppose $\lambda \in X(T)_+$ is a dominant root or is less than or equal to a fundamental weight. Assume that $p > 5$ if $\Phi$ is of type $E_8$ or $G_2$, and that $p > 3$ otherwise. Then the restriction maps
\[
\opH^1(G,L(\lambda)) \rightarrow \opH^1(\Gfq,L(\lambda)) \quad \text{and} \quad \opH^2(G,L(\lambda)) \rightarrow \opH^2(\Gfq,L(\lambda))
\]
are an isomorphism and an injection, respectively.
\end{theorem}

\begin{proof}
The theorem follows from Theorems \ref{theorem:vanishingconditionforExt1iso}, \ref{theorem:semisimplicity}, and Lemma \ref{lemma:Ext1weights}.
\end{proof}

\begin{remark}
If one restricts attention to dominant weights $\lambda \in X(T)_+$ that are less than or equal to a fundamental dominant weight, then Corollary \ref{corollary:socleofUrcohomology}, Theorem \ref{theorem:semisimplicity}, and Lemma \ref{lemma:Ext1weights} continue to hold, with exactly the same proofs, if $\Phi$ is of type $A_n$ or $D_n$, $p=3$, and $q > 3$. Thus, the methods of this paper yield a new proof of \cite[Theorem 1.2.1]{UGA:2011} for these cases also.
\end{remark}

\subsection{Weight spaces in second cohomology for \texorpdfstring{$U_1$}{U1}}

Suppose $\lambda \in X(T)_+$ is a dominant root or is less than or equal to a fundamental weight. We are now interested in finding bounds on $p$ and $q = p^r$ similar to those in Theorem \ref{theorem:H1Gfqiso} for which the restriction map $\opH^2(G,L(\lambda)) \rightarrow \opH^2(\Gfq,L(\lambda))$ is an isomorphism. For this we analyze the weights of the cohomology group $\Ext_{U_r}^2(k,L(\lambda))$. We begin by analyzing the space of $\Tfq$-invariants in $\Ext_{U_1}^2(k,L(\lambda))$.

\begin{lemma} \label{lemma:Ext2U1Tfqinvariants}
Suppose $\lambda \in X(T)_+$ is a dominant root or is less than or equal to a fundamental dominant weight. Assume that $p > 3$, and that $q > 5$ if $\Phi$ is of type $E_7$, $E_8$, $F_4$, or $G_2$. Assume also that $q > 5$ if $\Phi$ is of type $C_n$ and $\lambda \in \Z\Phi$ (i.e., if $\lambda$ is the zero weight, a root, or a fundamental weight indexed by an even integer). If $q = p$, assume that $\lambda$ is not one of the weights listed in Table \ref{table:problemweights}. Then $\Ext_{U_1}^2(k,L(\lambda))^{\Tfq} = \Ext_{U_1}^2(k,L(\lambda))^T$.
\end{lemma}

\begin{proof}
By \cite[I.9.20]{Jantzen:2003}, there exists a spectral sequence of $B$-modules
\begin{equation} \label{eq:U1Liealgspecseq}
E_2^{2i,j} = S^i(\fu^*)^{(1)} \otimes \opH^{j}(\fu,L(\lambda)) \Rightarrow \Ext_{U_1}^{2i+j}(k,L(\lambda)),
\end{equation}
with $E_2^{i,j}=0$ if $i$ is odd. Then $\Ext_{U_1}^2(k,L(\lambda))$ is a $T$-module subquotient of $E_2^{2,0} \oplus E_2^{0,2}$. To prove the lemma, it suffices to show that $(E_2^{2,0} \oplus E_2^{0,2})^{\Tfq} = (E_2^{2,0} \oplus E_2^{0,2})^T$.

From Table \ref{table:problemweights} and \cite[\S 7.1]{UGA:2011}, one has for all $\gamma \in \Delta$ that $(\lambda,\gamma^\vee) \leq 3$. In particular, $\lambda \in X_1(T)$, so $\opH^0(\fu,L(\lambda)) = \Hom_{\fu}(k,L(\lambda)) \cong L(\lambda)_{w_0\lambda}$. Then $E_2^{2,0}$ is isomorphic as a $B$-module to $(\fu^*)^{(1)} \otimes w_0\lambda$, which has weights of the form $p\beta - \lambda^*$ for $\beta \in \Phi^+$. Let $\gamma \in \Delta$. Recall that $\abs{(\beta,\gamma^\vee)} \leq 3$ if $\Phi$ is of type $G_2$, and that $\abs{(\beta,\gamma^\vee)} \leq 2$ otherwise. Then
\[
\abs{(p\beta - \lambda^*,\gamma^\vee)} \leq p \abs{(\beta,\gamma^\vee)} + (\lambda^*,\gamma^\vee) \leq 3p+3.
\]
Since $p > 3$, this expression is strictly less than $p^r-1$ if $r \geq 2$. It follows then for $r \geq 2$ that $p\beta - \lambda^* \in (p^r-1)X(T)$ only if $p\beta - \lambda^* = 0$, and hence that $(E_2^{2,0})^{\Tfq} = (E_2^{2,0})^T$. Now suppose $r = 1$, so that $q = p$, and suppose that $p\beta - \lambda^* \in (p-1)X(T)$. Then also $\beta - \lambda^* \in (p-1)X(T)$. Using the previous estimates for $\abs{(\beta,\gamma^\vee)}$, and using Table \ref{table:problemweights} and \cite[\S 7.1]{UGA:2011} for the possible values of $(\lambda^*,\gamma^\vee)$, one has
\begin{align*}
-4 &\leq (\beta-\lambda^*,\gamma^\vee) \leq 2 \quad \text{if $\Phi \neq E_8,G_2$,} \\
-5 &\leq (\beta-\lambda^*,\gamma^\vee) \leq 2 \quad \text{if $\Phi = E_8$, and} \\
-4 &\leq (\beta-\lambda^*,\gamma^\vee) \leq 3 \quad \text{if $\Phi = G_2$.}
\end{align*}
It then follows for $p = q > 5$ that $\beta - \lambda^* \in (p-1)X(T)$ only if $\beta - \lambda^* = 0$, that is, only if $\lambda = \beta^*$ is a dominant root. By assumption, $\lambda$ is not equal to a dominant root when $p = q$, so we conclude for $p = q > 5$ that $(E_2^{2,0})^{\Tfq} = (E_2^{2,0})^T$. Finally, suppose $p = q = 5$. Then $\Phi$ is of classical type or of type $E_6$, and $-4 \leq (\beta - \lambda^*,\gamma^\vee) \leq 2$. Moreover, the extreme value $-4$ is obtained only if $(\beta,\gamma^\vee) = -2$ and $(\lambda^*,\gamma^\vee) = 2$. Since by assumption $\lambda$ is not equal to a dominant root when $q = p$, we must have $\beta - \lambda^* \neq 0$, and consequently there must exist some $\gamma \in \Delta$ for which $(\beta,\gamma^\vee) = -2$ and $(\lambda^*,\gamma^\vee) = 2$. Since $\Phi$ is of classical type or of type $E_6$, this implies from the list of possible values for $\lambda$ that $\Phi$ is of type $B_2$ or $C_n$ and that $\lambda^* = \wtalpha$, hence that $\lambda = \wtalpha$. This again contradicts the assumptions on $\lambda$ when $p=q$, so we conclude that $(E_2^{2,0})^{\Tfq} = (E_2^{2,0})^T$.

Now consider the space $E_2^{0,2} = \opH^2(\fu,L(\lambda))$. First, $\opH^2(\fu,L(\lambda))$ identifies with a $T$-submodule of $\opH^2(\fu,k) \otimes L(\lambda)$ by \cite[Proposition 2.5.1]{UGA:2009}. Since $p > 3$, it then follows from \cite[Theorem 4.4]{Bendel:2007} that every weight of $\opH^2(\fu,L(\lambda))$ has the form $-w \cdot 0 + \nu$ for some $w \in W$ with $\ell(w) = 2$ and some weight $\nu$ of $L(\lambda)$. So let $-w \cdot 0 + \nu$ be a weight of this form, and suppose $-w \cdot 0 + \nu \in (p^r-1)X(T)$, say $-w \cdot 0 + \nu = (p^r-1)\sigma$ with $\sigma \in X(T)$. The term $-w \cdot 0$ is a sum of two linearly independent roots, so $(p^r-1) \sigma = \beta_1 + \beta_2 + \nu$ for some $\beta_1,\beta_2 \in \Phi$ with $\beta_1 \neq \pm \beta_2$. Conjugating this equation by an appropriate element of $W$ if necessary, we may assume that $\nu \in X(T)_+$. Since the weights of $L(\lambda)$ are stable under $W$, even after conjugating we still have $\nu \leq \lambda$, so that $\nu$ is either a dominant root or is a dominant weight less than or equal to a fundamental weight. With these assumptions on $\beta_1$, $\beta_2$ and $\nu$, we have verified by direct calculation that if $\Phi$ is of exceptional type, then the assumptions on $p$ and $q$ imply that the only solution to the equation $\beta_1 + \beta_2 + \nu = (p^r-1)\sigma$ is $\sigma = 0$.\footnote{We verified this assertion via an exhaustive search using the computer program GAP \cite{GAP4}. Our GAP code can be downloaded as an ancillary file from this paper's arXiv preprint page. The code is also available on the web at \href{http://www.math.uga.edu/~nakano/vigre/vigre.html}{http://www.math.uga.edu/$\sim$nakano/vigre/vigre.html}.} Thus $(E_2^{0,2})^{\Tfq} = (E_2^{0,2})^T$ if $\Phi$ is of exceptional type.

It remains to show that $(E_2^{0,2})^{\Tfq} = (E_2^{0,2})^T$ when $\Phi$ is of classical type. We first show this is true for $q = p^r > 5$. Suppose $\beta_1 + \beta_2 + \nu = (p^r-1)\sigma$ with $\nu \in X(T)_+$, $\nu \leq \lambda$ as above. Suppose for the moment that $(\nu,\gamma^\vee) \leq 1$ for all $\gamma \in \Delta$, that is, that $\nu$ is a $2$-restricted dominant weight. Then
\[
(p^r-1) \abs{(\sigma,\gamma^\vee)} \leq \abs{(\beta_1,\gamma^\vee)} + \abs{(\beta_2,\gamma^\vee)} + (\nu,\gamma^\vee) \leq 5.
\]
Thus, if $\nu$ is a $2$-restricted weight and $q = p^r>5$, then $(\sigma,\gamma^\vee) = 0$ for all $\gamma \in \Delta$, and hence $\sigma = 0$. On the other hand, if $\nu$ is not a $2$-restricted weight, then from Table \ref{table:problemweights} and \cite[\S 7.1]{UGA:2011} we see that necessarily $\nu = \lambda = \wtalpha$, and $\Phi$ must be of type $A_1$, $B_2$, or $C_n$. For type $A_1$ one has $E_2^{0,2} = \opH^2(\fu,L(\wtalpha)) = 0$ because $\dim \fu = 1$. For types $B_2$ and $C_n$ and $q = p^r > 5$, one can check (e.g., using the $\ve$-basis construction of $\Phi$ \cite[\S 12.1]{Humphreys:1978}) for all pairs of non-proportional roots $\beta_1,\beta_2 \in \Phi$ that $\beta_1 + \beta_2 + \wtalpha \in (p^r-1)X(T)$ only if $\beta_1 + \beta_2 + \wtalpha = 0$. Thus, $(E_2^{0,2})^{\Tfq} = (E_2^{0,2})^T$ if $\Phi$ is of classical type and $q = p^r > 5$.

Finally, suppose $\Phi$ is of classical type and $p = q = 5$. Then by assumption $\lambda$ is not one of the weights listed in Table \ref{table:problemweights}, so $\lambda$ must be less than or equal to a fundamental weight. In addition, if $\Phi = B_2$ then $\lambda = \omega_2$, and if $\Phi$ is of type $C_n$, then $\lambda = \omega_j$ with $j$ odd. We show under these conditions that a weight $-w\cdot 0 + \nu$ of $\opH^2(\fu,L(\lambda))$ is an element of $4X(T)$ only if it is zero. First suppose $\Phi$ is of type $C_n$. Since $\lambda = \omega_j$ with $j$ odd, one has $\lambda \notin \Z\Phi$, hence $\nu \notin \Z\Phi$ for all weights $\nu$ of $L(\lambda)$. Then for all $w \in W$ and all weight $\nu$ of $L(\lambda)$, one has $-w \cdot 0 + \nu \notin 4X(T) \subseteq \Z\Phi$.

Next suppose $\Phi = A_n$, so that $G=SL_{n+1}$. We may assume that $n \geq 3$, since $\opH^2(\fu,L(\wtalpha)) = 0$ if $\Phi$ is of type $A_1$, and since all of the fundamental weights and dominant roots $A_2$ are excluded from consideration when $q = 5$ by the assumptions on $\lambda$. Let $V$ be the natural $(n+1)$-dimensional representation of $SL_{n+1}$. Then for $1 \leq j \leq n$, $L(\omega_j)  = V(\omega_j) \cong \Lambda^j(V)$, the $j$-th exterior power of $V$. From this explicit description of $L(\omega_j)$, one can explicitly write down the weights $-w \cdot 0 + \nu$ for $w \in W$ with $\ell(w)=2$ and $\nu$ a weight of $L(\omega_j)$, and check for $n \geq 3$ that $-w \cdot 0 + \nu \notin 4X(T)$.

The cases $\Phi = B_n,D_n$ are handled similarly to type $A_n$. For example, for $\Phi = D_n$, let $V$ be the natural $2n$-dimensional representation for $SO_{2n}$. Then for for $1 \leq j \leq n-2$ one has $L(\omega_j) = V(\omega_j) \cong \Lambda^j(V)$ \cite[II.8.21]{Jantzen:2003}, and one can then check for each weight $\nu$ of $L(\omega_j)$ and for each $w \in W$ with $\ell(w) = 2$ that $-w \cdot 0 + \nu \in 4X(T)$ only if $-w \cdot 0 + \nu = 0$. In particular, this holds if $\nu = 0$, since $0$ is a weight of $\Lambda^2(V)$, so the condition $\opH^2(\fu,L(\lambda))^{\Tfq} = \opH^2(\fu,L(\lambda))^T$ also holds if $\lambda = 0$. For $j \in \set{n-1,n}$, the weight $\omega_j$ is minuscule, hence not in the root lattice. Consequently, if $\nu$ is a weight of $L(\omega_j)$, then $\nu \notin \Z\Phi$, and  $-w \cdot 0 + \nu \notin 4X(T) \subseteq \Z\Phi$. The details for $\Phi = B_n$ are similar, so are left to the reader.
\end{proof}

\begin{remark} \label{remark:nonzeroTfpweight}
Retain the assumptions on $p$ and $q$ made in Lemma \ref{lemma:Ext2U1Tfqinvariants}. Suppose $q = p$, and let $\lambda \in \set{\alpha_0,\wtalpha}$. Assume also that $p > 5$ if $\Phi$ is of type $A_3$, $B_2$, or $C_n$. Then one can argue as in the proof of the lemma to show that the condition $\opH^2(\fu,L(\lambda))^{\Tfp} = \opH^2(\fu,L(\lambda))^T$ holds in this case as well. The assumption $p > 5$ is necessary for types $A_3$, $B_2$, and $C_n$, since for these types one has the identities
\begin{align*}
-s_{\alpha_2} s_{\alpha_1} \cdot 0 + (-\alpha_3) &= 4\omega_2 - 4\omega_3 & \Phi = A_3, \\
-s_{\alpha_2} s_{\alpha_3} \cdot 0 + (-\alpha_1) &= 4 \omega_2 -4\omega_1 & \Phi = A_3, \\
-s_{\alpha_1} s_{\alpha_2} \cdot 0 + (-\alpha_2) &= 4 \omega_1 - 4 \omega_2 & \Phi = B_2,\\
-s_{\alpha_{n-1}} s_{\alpha_n} \cdot 0 + (\alpha_{n-1} + \alpha_n) &= 4\omega_{n-1} - 4 \omega_{n-2} & \Phi = C_n.
\end{align*}
On the other hand, the nonzero $\Tfp$-invariant weight $(p-1)\lambda$ occurs with multiplicity one in the $E_2^{2,0}$-term of \eqref{eq:U1Liealgspecseq}, and the proof of the lemma shows that this is the only nonzero $\Tfp$-invariant weight in $E_2^{2,0}$. Observe that $E_2^{0,1} = \opH^1(\fu,L(\lambda))$ is a $T$-module subquotient of $\opH^1(\fu,k) \otimes L(\lambda)$, hence has weights of the form $\alpha + \nu$ for $\alpha \in \Delta$ and $\nu$ a weight of $L(\lambda)$. In particular, $\nu \leq \lambda$. Then $(p-1)\lambda$ occurs as a weight of $E_2^{0,1}$ only if $(p-1)\lambda \leq \alpha + \lambda$ for some $\alpha \in \Delta$. Equivalently, $(p-2)\lambda \leq \alpha$. Since $\lambda \in \set{\alpha_0,\wtalpha}$ this condition is absurd, so $(p-1)\lambda$ is not a weight of $E_2^{0,1}$. Thus, if $\lambda \in \set{\alpha_0,\wtalpha}$, and if $p$ satisfies the conditions stated above, it follows that $\dim \Ext_{U_1}^2(k,L(\lambda))_{(p-1)\lambda} = 1$.

Now suppose $\Phi = A_2$, $p = q = 5$, and $\lambda \in \set{\omega_1,\omega_2}$. Arguing as in the proof of the lemma, one can check in this case that the only solutions to $-w \cdot 0 + \nu \in 4X(T)$ with $w \in W$, $\ell(w) = 2$, and $\nu$ a weight of $L(\lambda)$ are $-s_{\alpha_1} s_{\alpha_2} \cdot 0 + \omega_1 = 4\omega_1$ (if $\lambda = \omega_1$), and $-s_{\alpha_2} s_{\alpha_1} \cdot 0 + \omega_2 = 4\omega_2$ (if $\lambda = \omega_2$). It then follows that $\opH^2(\fu,L(\lambda))^{\Tfp}$ is at most one-dimensional, and if nonzero, is spanned by a vector of weight $4\lambda$. Kostant's Theorem (cf.\ \cite[Theorem 4.2.1]{UGA:2009} and the footnote in \cite[\S6.1]{UGA:2011}) apply to show that the weight $4\lambda$ does in fact occur in $E_2^{0,2} = \opH^2(\fu,L(\lambda))$. One can then analyze the weights of the low degree terms in \eqref{eq:U1Liealgspecseq}, again using Kostant's Theorem to determine the weights of $\opH^1(\fu,L(\lambda))$, to see that the $4\lambda$-weight space of $E_\infty^{0,2}$ is nonzero. The proof of the lemma also shows in this case that $(E_2^{2,0})^{\Tfp} = (E_2^{2,0})^T$, so it follows that $\dim \Ext_{U_1}^2(k,L(\lambda))_{(p-1)\lambda} = 1$, and that $(p-1)\lambda$ is the only nonzero $\Tfp$-invariant weight in $\Ext_{U_1}^2(k,L(\lambda))$.
\end{remark}

\subsection{Weight spaces in second cohomology for \texorpdfstring{$U_r$, $r \geq 2$}{Ur, r > 1}}

We now analyze the space of $\Tfq$-invariants in $\Ext_{U_r}^2(k,L(\lambda))$ for $r \geq 2$.

\begin{lemma} \label{lemma:Ext2UrTfqinvariants}
Suppose $\lambda \in X(T)_+$ is a dominant root or is less than or equal to a fundamental weight. Assume $p > 5$ if $\Phi$ is of type $E_8$ or $G_2$, and that $p > 3$ otherwise. Set $q = p^r$ with $r \geq 2$. Then $\Ext_{U_r}^2(k,L(\lambda))^{\Tfq} = \Ext_{U_r}^2(k,L(\lambda))^T$ provided $\lambda$ is not one of the weights listed in Table \ref{table:problemweights}.
\end{lemma}

\begin{proof}
Consider the LHS spectral sequence for $U_r$ and its normal subgroup scheme $U_1$:
\begin{equation} \label{eq:LHSforUrTfqinvariants}
E_2^{i,j} = \Ext_{U_r/U_1}^i(k,\Ext_{U_1}^j(k,L(\lambda))) \Rightarrow \Ext_{U_r}^{i+j}(k,L(\lambda)).
\end{equation}
To prove the lemma, it suffices to show $(E_2^{0,2} \oplus E_2^{1,1} \oplus E_2^{2,0})^{\Tfq} = (E_2^{0,2} \oplus E_2^{1,1} \oplus E_2^{2,0})^T$. First, $E_2^{0,2} \cong \Ext_{U_1}^2(k,L(\lambda))^{U_r/U_1}$, so Lemma \ref{lemma:Ext2U1Tfqinvariants} implies that $(E_2^{0,2})^{\Tfq} = (E_2^{0,2})^T$. Next, $\Ext_{U_1}^1(k,L(\lambda))$ is semisimple as a $B/U_1$-module by Theorem \ref{theorem:semisimplicity}, hence trivial as a $U_r/U_1$-module, so there exists a $B$-module isomorphism
\[
E_2^{1,1} \cong \Ext_{U_r/U_1}^1(k,k) \otimes \Ext_{U_1}^1(k,L(\lambda)) \cong \opH^1(U_{r-1},k)^{(1)} \otimes \Ext_{U_1}^1(k,L(\lambda))
\]
(cf.\ Remark \ref{remark:twistedschemes}). Then by Lemma \ref{lemma:weightsofH1Ur} and Theorem \ref{theorem:semisimplicity}, each weight of $E_2^{1,1}$ can be written as either $p^i\beta - s_\alpha \cdot \lambda^*$, $p^i \beta - (\lambda^* - p^n \alpha)$, or $p^i \beta - \sigma$ for some $\alpha,\beta \in \Delta$, some integers $1 \leq i,n < r$, and some $\sigma \in X(T)_+$ with $\sigma < \lambda^*$. First consider a weight of the form $p^i \beta - s_\alpha \cdot \lambda^*$. Given $\gamma \in \Delta$, one has $\abs{p^i\beta - s_\alpha \cdot \lambda^*,\gamma^\vee)} \leq p^i \abs{(\beta,\gamma^\vee)} + \abs{(s_\alpha \cdot \lambda^*,\gamma^\vee)}$. Then one can verify as in the proof of Lemma \ref{lemma:Ext1weights} the following inequalities:
\begin{align*}
\abs{(p^i\beta - s_\alpha \cdot \lambda^*,\gamma^\vee)} &\leq 2p^{r-1} + 4 \quad \text{if $\Phi \in \set{D_n,E_6}$,} \\
\abs{(p^i\beta - s_\alpha \cdot \lambda^*,\gamma^\vee)} &\leq 2p^{r-1} + 6 \quad \text{if $\Phi \in \set{A_n,B_n,C_n,E_7,F_4}$,} \\
\abs{(p^i\beta - s_\alpha \cdot \lambda^*,\gamma^\vee)} &\leq 2p^{r-1} + 8 \quad \text{if $\Phi = E_8$, and} \\
\abs{(p^i\beta - s_\alpha \cdot \lambda^*,\gamma^\vee)} &\leq 3p^{r-1} + 6 \quad \text{if $\Phi = G_2$.}
\end{align*}
Since $p > 3$, one has $\abs{(p^i\beta - s_\alpha \cdot \lambda^*,\gamma^\vee)} < p^r-1$, so it follows that $p^i \beta - s_\alpha \cdot \lambda^* \in (p^r-1)X(T)$ only if $p^i \beta - s_\alpha \cdot \lambda^* =0$. A similar analysis shows that the weights of the forms $p^i\beta - (\lambda^* - p^n\alpha)$ and $p^i\beta - \sigma$ are elements of $(p^r-1)X(T)$ only if they are zero, so it follows that $(E_2^{1,1})^{\Tfq} = (E_2^{1,1})^T$.

Finally, consider the weights of
\[
E_2^{2,0} = \Ext_{U_r/U_1}^2(k,\Hom_{U_1}(k,L(\lambda))) \cong \opH^2(U_{r-1},k)^{(1)} \otimes w_0\lambda.
\]
It follows from \cite[I.9.14]{Jantzen:2003} that each weight of $E_2^{2,0}$ can be written in one of the following forms:
\begin{align*}
& p^a \alpha + p^b  \beta - \lambda^* && \alpha,\beta \in \Phi^+, 1 \leq a < b \leq r-1 \text{ (if $r \geq 3$),} \\
& p^e(\alpha + \beta) - \lambda^* && \alpha,\beta \in \Phi^+, 1 \leq e \leq r-1, \text{ or} \\
& p^c \alpha - \lambda^* && \alpha \in \Phi^+, 2 \leq c \leq r.
\end{align*}
Again, an elementary case-by-case analysis like that conducted above for the $E_2^{1,1}$-term shows that any weight of the form $p^a\alpha + p^b \beta - \lambda^*$ or $p^e(\alpha+\beta) - \lambda^*$ is an element of $(p^r-1)X(T)$ only if it is the zero weight. An elementary analysis like that performed in the proof of Lemma \ref{lemma:Ext2U1Tfqinvariants} for the $E_2^{2,0}$-term of the spectral sequence \eqref{eq:U1Liealgspecseq} also shows that a weight of the form $p^c \alpha - \lambda^*$ is an element of $(p^r-1)X(T)$ only if $c = r$ and $\lambda = \alpha^*$ is a dominant root. Thus, if $\lambda$ is not one of the weights listed in Table \ref{table:problemweights}, then $(E_2^{2,0})^{\Tfq} = (E_2^{2,0})^T$.
\end{proof}

\begin{remark} \label{remark:nonzeroTfqweight}
Retain the assumptions on $p$, $q$, and $r$ from the lemma, and let $\lambda$ be a dominant root. Then the proof of the lemma shows that the only possible nonzero $\Tfq$-invariant weight in $E_2^{2,0}$, hence the only possible nonzero $\Tfq$-invariant weight in $\Ext_{U_r}^2(k,L(\lambda))$, is $(p^r-1)\lambda$. Since $\Ext_{U_1}^1(k,L(\lambda))^{\Tfq} = \Ext_{U_1}^1(k,L(\lambda))^T$ by Lemma \ref{lemma:Ext1weights}, it follows from considering the low degree terms of the spectral sequence \eqref{eq:LHSforUrTfqinvariants} that
\[
\dim \Ext_{U_r}^2(k,L(\lambda))_{(p^r-1)\lambda} = \dim \, (E_2^{2,0})_{(p^r-1)\lambda}.
\]
This common dimension is also the dimension of the $p^r\lambda$-weight space in $\opH^2(U_{r-1},k)^{(1)}$, or equivalently, that of the $p^{r-1}\lambda$-weight space in $\opH^2(U_{r-1},k)$. Since $\opH^2(B_{r-1},k) = \opH^2(U_{r-1},k)^{T_{r-1}}$, we conclude from \cite[Proposition 5.4 and Theorem 5.3]{Bendel:2007} that $\dim \Ext_{U_r}^2(k,L(\lambda))_{(p^r-1)\lambda} = 1$, and this is the only nonzero $\Tfq$-invariant weight in $\Ext_{U_r}^2(k,L(\lambda))$.
\end{remark}

\subsection{An isomorphism for second cohomology} \label{subsection:secondcohomology}

The weight calculations of the previous section enable us to complete our program for understanding the restriction map $\opH^2(G,L(\lambda)) \rightarrow \opH^2(\Gfq,L(\lambda))$.

\begin{theorem} \label{theorem:H2Gfpiso}
Let $\lambda \in X(T)_+$ with $\lambda \leq \omega_j$ for some $1 \leq j \leq n$. Assume that $p > 3$, and that $p > 5$ if $\Phi$ is of type $E_7$, $E_8$, $F_4$, or $G_2$, or if $\Phi$ is of type $C_n$ and $\lambda \in \Z\Phi$ (i.e., if $\lambda$ is the zero weight or is a fundamental weight indexed by an even integer). Assume that $\lambda$ is not one of the weights listed in Table \ref{table:problemweights}. Then the restriction map
\[
\opH^2(G,L(\lambda)) \rightarrow \opH^2(\Gfp,L(\lambda))
\]
is an isomorphism.
\end{theorem}

\begin{proof}
This follows from Theorem \ref{theorem:Ext2Gfqiso} by applying Theorem \ref{theorem:semisimplicity} and Lemmas \ref{lemma:Ext1weights} and \ref{lemma:Ext2U1Tfqinvariants}.
\end{proof}

We obtain a slightly better result when $r>1$.

\begin{theorem} \label{theorem:H2Gfqiso}
Let $\lambda \in X(T)_+$ with $\lambda \leq \omega_j$ for some $1 \leq j \leq n$. Assume that $p > 5$ if $\Phi$ is of type $E_8$ or $G_2$, and $p > 3$ otherwise. Set $q = p^r$ with $r \geq 2$, and assume that $\lambda$ is not one of the weights listed in Table \ref{table:problemweights}. Then the restriction map
\[
\opH^2(G,L(\lambda)) \rightarrow \opH^2(\Gfq,L(\lambda))
\]
is an isomorphism.
\end{theorem}

\begin{proof}
This follows from Theorem \ref{theorem:Ext2Gfqiso} by applying Theorem \ref{theorem:semisimplicity} and Lemmas \ref{lemma:Ext1weights} and \ref{lemma:Ext2UrTfqinvariants}.
\end{proof}

\subsection{Proof of Theorem \ref{theorem:Ext2restrictioniso}} \label{subsection:proofofExt2restrictioniso}

We can now prove Theorem \ref{theorem:Ext2restrictioniso}.

\begin{proof}[Proof of Theorem \ref{theorem:Ext2restrictioniso}]
The result now follows from Theorems \ref{theorem:H2Gfpiso} and \ref{theorem:H2Gfqiso}.
\end{proof}

\subsection{Vanishing results} \label{subsection:vanishingresults}

Let $\lambda \in X(T)_+$ with $\lambda \leq \omega_j$ for some $1 \leq j \leq n$, and assume that $\lambda$ is not one of the weights listed in Table \ref{table:problemweights}. Then under the conditions of Table \ref{table:pandq}, Theorems \ref{theorem:H1Gfqiso}, \ref{theorem:H2Gfpiso}, and \ref{theorem:H2Gfqiso} reduce the problems of computing the cohomology groups $\opH^1(\Gfq,L(\lambda))$ and $\opH^2(\Gfq,L(\lambda))$ for $\Gfq$ to those of computing the corresponding rational cohomology groups for the full algebraic group $G$, where results are typically easier to obtain. The first cohomology group $\opH^1(G,L(\lambda))$ has been treated already in \cite[\S\S 4--6]{UGA:2011}, so we concentrate now $\opH^2(G,L(\lambda))$.

\begin{theorem} \label{theorem:H=Vornotlinked}
Let $\lambda \in X(T)_+$ with $\lambda \leq \omega_j$ for some $1 \leq j \leq n$. Suppose $p$ and $q$ satisfy the conditions in Table \ref{table:pandq}, and $\lambda$ is not one of the weights listed in Table \ref{table:problemweights}. Then $\opH^2(\Gfq,L(\lambda)) = 0$ if either $L(\lambda) = H^0(\lambda)$, or if $\lambda$ is not linked to zero under the dot action of the affine Weyl group. 
\end{theorem}

\begin{proof}
The result follows from Theorem \ref{theorem:Ext2restrictioniso} by \cite[II.4.13]{Jantzen:2003} and the Linkage Principle.
\end{proof}

\begin{corollary} \label{corollary:H2vanishingclassicalandminuscule}
Let $\lambda \in X(T)_+$ with $\lambda \leq \omega_j$ for some $1 \leq j \leq n$. Suppose $p$ and $q$ satisfy the conditions listed in Table \ref{table:pandq}, and $\lambda$ is not one of the weights listed in Table \ref{table:problemweights}. If $\Phi$ is of type $C_n$, assume that $\lambda = \omega_j$ with $j$ odd, and if $\Phi$ is of exceptional type, assume that $\lambda$ is a minuscule dominant weight. Then $\opH^2(\Gfq,L(\lambda)) \cong \opH^2(G,L(\lambda)) = 0$.
\end{corollary}

\begin{proof}
First recall from \cite[\S 7.1]{UGA:2011} that if $\Phi$ is of classical type, then the condition $\lambda \leq \omega_j$ implies that either $\lambda = 0$ or $\lambda = \omega_i$ for some $1 \leq i \leq j$. Also, if $\lambda$ is a minuscule dominant weight, then $\lambda$ is minimal with respect to the partial order $\leq$ on $X(T)_+$. This combined with \cite[II.8.21]{Jantzen:2003} shows that if $\lambda$ is minuscule, or if $\Phi$ is of type $A_n$, $B_n$, or $D_n$, then $L(\lambda) = H^0(\lambda) = V(\lambda)$. On the other hand, if $\Phi$ is of type $C_n$ and $\lambda = \omega_j$ with $j$ odd, then $\lambda$ is not an element of the root lattice $\Z\Phi$, so it cannot be linked to zero under the dot action of the affine Weyl group. Now the result follows from Theorem \ref{theorem:H=Vornotlinked}.
\end{proof}

\begin{lemma} \label{lemma:largeprimevanishing}
Let $\lambda \in X(T)_+$ with $\lambda \leq \omega_j$ for some $1 \leq j \leq n$. Suppose $p$ and $q$ satisfy the conditions listed in Table \ref{table:pandq}, and $\lambda$ is not one of the weights listed in Table \ref{table:problemweights}. Set $h_\lambda = (\lambda,\alpha_0^\vee)$, where $\alpha_0$ is the highest short root in $\Phi$. Suppose $p \geq h+h_\lambda-1$. Then $\opH^2(\Gfq,L(\lambda)) \cong \opH^2(G,L(\lambda)) = 0$.
\end{lemma}

\begin{proof}
If $p \geq h+h_\lambda - 1$, then $\lambda \in \Czbar$, so $L(\lambda) = H^0(\lambda)$ by \cite[II.5.6]{Jantzen:2003}. Now $\opH^2(\Gfq,L(\lambda)) = 0$ by Theorem \ref{theorem:H=Vornotlinked}.
\end{proof}

\begin{proposition} \label{proposition:H2vanishingexceptional}
Let $\lambda \in X(T)_+$ with $\lambda \leq \omega_j$ for some $1 \leq j \leq n$. Suppose $p$ and $q$ satisfy the conditions listed in Table \ref{table:pandq}, and $\lambda$ is not one of the weights listed in Table \ref{table:problemweights}. Assume that $\Phi$ is of exceptional type, and that $p$ does not equal one of the primes listed next to $\lambda$ in the Hasse diagram for $\Phi$ in \cite[\S 7.1]{UGA:2011}. Then $\opH^2(\Gfq,L(\lambda)) \cong \opH^2(G,L(\lambda)) = 0$. The conclusion $\opH^2(\Gfq,L(\lambda)) = 0$ also holds in the following additional cases:
\begin{enumerate}
\item $\Phi$ is of type $E_7$, $p=7$, and $\lambda = \omega_6$.
\item $\Phi$ is of type $E_7$, $p=19$, and $\lambda = 2\omega_1$;
\item $\Phi$ is of type $E_8$, $p=7$, and $\lambda = \omega_3$;
\item $\Phi$ is of type $E_8$, $p=31$, and $\lambda = 2\omega_8$; and
\item $\Phi$ is of type $F_4$, $p=13$, and $\lambda = 2\omega_4$.
\end{enumerate}
\end{proposition}

\begin{proof}
If $p > 31$, then $\opH^2(G,L(\lambda))=0$ by Lemma \ref{lemma:largeprimevanishing}. If $5 < p \leq 31$, and if $p$ does not equal one of the primes listed next to $\lambda$ in the Hasse diagram for $\Phi$ in \cite[\S 7.1]{UGA:2011}, then $\lambda$ is not linked to zero under the dot action of the affine Weyl group, so $\opH^2(G,L(\lambda)) = 0$ by the Linkage Principle. If $\Phi$ is of type $E_7$, $p=7$, and $\lambda = \omega_6$, then $\rad_G V(\omega_6) \cong k$ (see the proof of \cite[Theorem 6.3.2]{UGA:2011}), so $\opH^2(G,L(\omega_6)) \cong \Ext_G^2(L(\omega_6),k) \cong \Ext_G^1(k,k)=0$ by \cite[II.4.13--14]{Jantzen:2003}. If $\Phi$ is of type $E_8$, $p=7$, and $\lambda = \omega_3$, then $L(\lambda) = H^0(\lambda)$ by \cite[\S 4.6]{Jantzen:1991}, so $\opH^2(G,L(\lambda)) = 0$ by \cite[II.4.13]{Jantzen:2003}. For the other cases, $\opH^2(G,L(\lambda)) = 0$ by the proof of \cite[Theorem 6.3.1]{UGA:2011}. Now the proposition follows from Theorem \ref{theorem:Ext2restrictioniso}.
\end{proof}

\subsection{\texorpdfstring{Analysis for type $E_{8}$ and $p=31$}{Analysis for type E8 and p=31}} \label{subsection:E8p31}

We now address the case $\Phi=E_8$, $p=31$, and $\lambda \in \set{\omega_6+\omega_8,\ \omega_7+\omega_8}$. Note here that $p \geq h$, where $h = 30$ is the Coxeter number of the root system $\Phi$. By Theorem \ref{theorem:Ext2restrictioniso}, to compute $\opH^2(\Gfq,L(\lambda))$ it suffices to compute the rational cohomology group $\opH^2(G,L(\lambda))$. For this we consider the structure of the induced modules $H^0(\lambda_i)$, with
\begin{align*}
\lambda_{1} &=0,\\
\lambda_{2} &=s_{0}\cdot 0=2\omega_{8},\\
\lambda_{3} &=(s_{0}s_{8})\cdot 0=\omega_{7}+\omega_{8}, \text{ and} \\
\lambda_{4} &=(s_{0}s_{8}s_{7})\cdot 0=\omega_{6}+\omega_{8}.
\end{align*}

Obviously, $H^0(0) = L(0) = k$. Next, the composition factors of $H^0(s_0 \cdot 0)$ have the form $L(\mu)$ with $\mu \leq s_0 \cdot 0$ and $\mu \in W_p \cdot 0$. From \cite[Figure 3]{UGA:2011}, we see that the only possible values for $\mu$ are $\mu = 0$ and $\mu = s_0 \cdot 0$. Taking $\lambda = 0$, $w = w_1 = 1$, and $s = s_0$ in \cite[II.7.18]{Jantzen:2003}, we see $[H^0(s_0 \cdot 0):L(0)] = [H^0(0):L(0)] = 1$, so $H^0(s_0 \cdot 0)$ is uniserial with socle $L(s_0 \cdot 0)$ and head $L(0)$. Now consider the structure of $H^0(s_0s_8 \cdot 0)$. Inspecting \cite[Figure~3]{UGA:2011} once more, the only possible composition factors of $H^0(s_0s_8 \cdot 0)$ are $L(s_0s_8 \cdot 0)$, $L(s_0 \cdot 0)$, and $L(0)$, though as explained in \cite[\S 6.4]{UGA:2011}, $L(0)$ cannot occur as a composition factor of $H^0(s_0s_8 \cdot 0)$. Then taking $\lambda = 0$, $w = w_1 = s_0$, and $s = s_8$ in \cite[II.7.18]{Jantzen:2003}, we see $[H^0(s_0s_8):L(s_0 \cdot 0)] = [H^0(s_0 \cdot 0):L(s_0 \cdot 0)] = 1$, so that $H^0(s_0s_8 \cdot 0)$ is uniserial with socle $L(s_0s_8 \cdot 0)$ and head $L(s_0 \cdot 0)$. Thus, the induced modules $H^0(\lambda_i)$ for $i \in \set{1,2,3}$ are all uniserial, and their composition series may be represented graphically as follows (diagrams are read from top to bottom, with the head on top and the socle at the bottom):


\begin{center}
\begin{tikzpicture}[scale=.8]
\tikzstyle{every node} = [fill=white, minimum size=2em]
\draw (-1.5,0) node {$H^0(0)$:};
\draw (0,0) node {$L(0)$};
\draw (3,0) node {$H^0(s_0 \cdot 0)$:};
\draw (5,1) node {$L(0)$} -- ++ (0,-2) node {$L(s_0 \cdot 0)$};
\draw (8.5,0) node {$H^0(s_0s_8 \cdot 0)$:};
\draw (11,1) node {$L(s_0 \cdot 0)$} -- ++ (0,-2) node {$L(s_0s_8 \cdot 0)$};
\end{tikzpicture}
\end{center}

Finally, consider the structure of $H^0(s_0s_8s_7 \cdot 0) = H^0(\omega_6+\omega_8)$. By \cite[II.7.19]{Jantzen:2003} with $\lambda = 0$, $\mu = -\omega_7$, $w = s_0s_8$, and $s = s_7$, there exists a short exact sequence
\begin{equation} \label{eq:sestranslation}
0 \rightarrow H^0(s_0s_8 \cdot 0) \rightarrow T_\mu^{0} H^0(s_0s_8 \cdot \mu) \rightarrow H^0(s_0s_8s_7 \cdot 0) \rightarrow 0.
\end{equation}
Moreover, $H^0(s_0s_8 \cdot \mu) = T_0^\mu H^0(s_0s_8 \cdot 0)$ by \cite[II.7.11]{Jantzen:2003}. Now $T_0^\mu L(s_0 \cdot 0) = 0$ by \cite[II.7.15]{Jantzen:2003}, because $s_0 \cdot \mu$ is in the lower closure (and not the upper closure) of the facet containing $s_0 \cdot 0$. Since the translation functor $T_0^\mu$ is exact, we conclude from \cite[II.7.14]{Jantzen:2003} and the composition series for $H^0(s_0s_8 \cdot 0)$ given above that $H^0(s_0s_8 \cdot \mu) = T_0^\mu H^0(s_0s_8 \cdot 0) = T_0^\mu L(s_0s_8 \cdot 0) = L(s_0s_8 \cdot \mu)$. In particular, this shows that $H^0(s_0s_8 \cdot \mu) = H^0(4\omega_8)$ is irreducible. Now by \cite[II.7.20]{Jantzen:2003}, the middle term $T_\mu^0 H^0(s_0s_8 \cdot \mu) = T_\mu^0 L(s_0s_8 \cdot \mu)$ in \eqref{eq:sestranslation} has head and socle isomorphic to $L(s_0s_8 \cdot 0)$, and these are the only occurrences of the composition factor $L(s_0s_8 \cdot 0)$ in $T_\mu^0 H^0(s_0s_8 \cdot \mu)$.

Set $M = \rad_G T_\mu^0 H^0(s_0s_8 \cdot \mu) / \soc_G T_\mu^0 H^0(s_0s_8 \cdot \mu)$, the \emph{heart} or \emph{Jantzen middle} of the module $T_\mu^0 H^0(s_0s_8 \cdot \mu)$. By \cite[II.7.20(b)]{Jantzen:2003}, the simple module $L(s_0s_8s_7 \cdot 0)$ occurs once as a composition factor of $M$, and by \cite[II.7.20(c)]{Jantzen:2003}, the remaining composition factors in $M$ must have the form $L(w' \cdot 0)$ with $w' \in W_p$, $w' \cdot 0 \in X(T)_+$, $w' \cdot 0 \neq s_0s_8 \cdot 0$, and $w's_7 \cdot 0 < w' \cdot 0$. Inspecting the short exact sequence \eqref{eq:sestranslation}, the composition series for $H^0(s_0s_8 \cdot 0)$, and \cite[Figure 3]{UGA:2011}, and recalling the discussion of \cite[\S 6.4]{UGA:2011}, we see that the only possible value for $w'$ is $w' = s_0$. We have $\Ext_G^1(L(s_0s_8 \cdot 0),L(s_0 \cdot 0)) = k$ by the calculation of the composition series for $H^0(s_0s_8 \cdot 0)$ (cf.\ \cite[II.2.14]{Jantzen:2003}), so by \cite[II.7.20(c)]{Jantzen:2003}, $L(s_0 \cdot 0)$ occurs once in $\hd_G M$, and once in $\soc_G M$. Using the fact that there are no self-extensions between simple modules, it follows then that either $M = L(s_0 \cdot 0) \oplus L(s_0s_8s_7 \cdot 0)$, or that $M$ has a composition series represented graphically as

\begin{center}
\begin{tikzpicture}[scale=.8]
\tikzstyle{every node} = [fill=white, minimum size=2em]
\draw (0,0) node {$M$:};
\draw (2.5,1.5) node {$L(s_0 \cdot 0)$} -- ++ (0,-1.5) node {$L(s_0 s_8 s_7 \cdot 0)$} -- ++ (0,-1.5) node {$L(s_0 \cdot 0)$};
\end{tikzpicture}
\end{center}

\subsection{Proof of Corollary \ref{corollary:H2vanishing}} \label{subsection:proofofH2vanishing}

We now complete the proof of Corollary \ref{corollary:H2vanishing}.

\begin{proof}[Proof of Corollary \ref{corollary:H2vanishing}]
Combining Corollary \ref{corollary:H2vanishingclassicalandminuscule}, Proposition \ref{proposition:H2vanishingexceptional}, and Lemma \ref{lemma:typeclargeprimevanishing} below, we obtain all of the cases of Corollary \ref{corollary:H2vanishing} except for $\Phi = E_8$, $p=31$, and $\lambda \in \set{\omega_6+\omega_8,\ \omega_7+\omega_8}$. We treat these two cases now. First, for $\lambda = \omega_7 + \omega_8 = s_0s_8 \cdot 0$ we have the short exact sequence
\[
0 \rightarrow L(s_0 s_8 \cdot 0) \rightarrow H^0(s_0 s_8 \cdot 0) \rightarrow L(s_0 \cdot 0) \rightarrow 0.
\]
Considering the associated long exact sequence in cohomology, and applying \cite[Theorem 1.2.3]{UGA:2011}, we get $\opH^2(G,L(s_0s_8 \cdot 0)) \cong \opH^1(G,L(s_0 \cdot 0)) \cong k$. Next, if the Lusztig Character Formula holds for type $E_8$ when $p=31$, then $\Ext_G^1(L(s_0s_8 \cdot 0),L(s_0s_8s_7 \cdot 0)) \neq 0$ by \cite[II.C.2(iii)]{Jantzen:2003}. This means that $L(s_0s_8 \cdot 0)$ must occur in the second socle layer of $H^0(s_0s_8s_7 \cdot 0)$. Consequently, the second possible composition series for $M$ given above is ruled out, so there exists a short exact sequence
\[
0 \rightarrow L(s_0s_8s_7 \cdot 0) \rightarrow H^0(s_0s_8s_7 \cdot 0) \rightarrow L(s_0s_8 \cdot 0) \rightarrow 0.
\]
Then considering the associated long sequence in cohomology, and applying \cite[Theorem 1.2.3]{UGA:2011}, one gets $\opH^2(G,L(s_0s_8s_7 \cdot 0)) \cong \opH^1(G,L(s_0s_8 \cdot 0)) = 0$. Now apply Theorem \ref{theorem:Ext2restrictioniso}.
\end{proof}

\subsection{Proof of Theorem \ref{theorem:otherH2calculations}} \label{subsection:additionalcalculations}

For the proof of Theorem \ref{theorem:otherH2calculations} we require the following lemmas:

\begin{lemma} \label{lemma:typeBlinkage}
Let $\Phi$ be of type $B_n$ with $n \geq 3$. Suppose $p > 3$, and that $n \not\equiv 1 \mod p$. Then $\alpha_0$ and $\wtalpha$ are not linked under the dot action of the affine Weyl group $W_p = W \ltimes p\Z\Phi$. The weights $\alpha_0$ and $\wtalpha$ are also not linked under the dot action of $W_p$ if $\Phi$ is of type $B_2$ and $p > 2$.
\end{lemma}

\begin{proof}
Suppose $\Phi$ is of type $B_n$ with $n \geq 3$ and $\alpha_0 = \omega_1$ and $\wtalpha = \omega_2$ are linked under the dot action of the affine Weyl group $W_p$. Then there exists $w \in W$ such that $w(\alpha_0 + \rho) \equiv \wtalpha + \rho \mod p\Z\Phi$. In terms of the $\ve$-basis for $\Phi$ \cite[\S 12.1]{Humphreys:1978}, this congruence equation can be rewritten as
\begin{equation} \label{eq:typeBcongruence} \textstyle
w(\frac{2n+1}{2}, \frac{2n-3}{2}, \frac{2n-5}{2}, \ldots, \frac{3}{2}, \frac{1}{2}) \equiv (\frac{2n+1}{2}, \frac{2n-1}{2}, \frac{2n-5}{2}, \ldots, \frac{3}{2}, \frac{1}{2}) \mod p\Z\Phi,
\end{equation}
where we have written $(a_1,\ldots,a_n)$ to denote $\sum_{i=1}^n a_i \ve_i$. In terms of the $\ve$-basis, $\Z\Phi$ consists of all integral combinations of the $\ve_i$, and the Weyl group acts as all permutations and sign changes of the set $\set{\ve_1,\ldots,\ve_n}$, so we can interpret \eqref{eq:typeBcongruence} as a system of congruences in $\Z[\frac{1}{2}]$ modulo $p\Z$. Thus, it makes sense to multiply both sides of \eqref{eq:typeBcongruence} by $2$, to obtain the new congruence equation
\begin{equation} \label{eq:typeBcongruence2}
w(2n+1,2n-3,2n-5,\ldots,3,1) \equiv (2n+1,2n-1,2n-5,\ldots,3,1) \mod p\Z^n.
\end{equation}
Here we consider both sides of the equation as elements of $\Z^n$, with $W$ acting on $\Z^n$ via place permutations and sign changes. There are now eight possibilities:
\begin{itemize}
\item $2n+1 \equiv \pm (2n+1) \mod p$ and $2n-1 \equiv \pm (2n-3) \mod p$, or
\item $2n+1 \equiv \pm (2n-3) \mod p$ and $2n-1 \equiv \pm (2n+1) \mod p$.
\end{itemize}
The only pair of congruence equations that does not contradict the assumption $p > 3$ is $2n+1 \equiv 2n+1 \mod p$ and $2n-1 \equiv -(2n-3) \mod p$. This pair of congruence equations is equivalent to the condition $n \equiv 1 \mod p$, which contradicts the assumption on $n$. Thus, if $p > 3$ and if $n \not\equiv 1 \mod p$, then $\alpha_0$ is not linked to $\wtalpha$ under the dot action of $W_p$. The statement for type $B_2$ can be similarly verified.
\end{proof}

\begin{lemma} \label{lemma:typeClinkage}
Let $\Phi$ be of type $C_n$. Suppose $p > 3$. Then $\alpha_0$ and $\wtalpha$ are not linked under the dot action of the extended affine Weyl group $\wt{W}_p = W \ltimes pX(T)$. In particular, $\alpha_0$ and $\wtalpha$ are not linked under the dot action of the ordinary affine Weyl group $W_p$.
\end{lemma}

\begin{proof}
The proof is similar to that of Lemma \ref{lemma:typeBlinkage}, so is left to the reader.
\end{proof}

We now complete the proof of Theorem \ref{theorem:otherH2calculations}.

\begin{proof}[Proof of Theorem \ref{theorem:otherH2calculations}]
The strategy of the proof is as follows. First, we verify that the hypotheses of Theorem~\ref{theorem:alternatecalculation} are satisfied with $\sigma = \lambda$, so that $\opH^2(\Gfq,L(\lambda)) \cong \Ext_G^2(V(\lambda)^{(r)},L(\lambda) \otimes H^0(\lambda))$. Next, we observe by Remarks \ref{remark:boundonExt2dimension}, \ref{remark:nonzeroTfpweight}, and \ref{remark:nonzeroTfqweight} that
\[
\Ext_G^2(V(\lambda)^{(r)},L(\lambda) \otimes H^0(\lambda)) \cong \Hom_{B/B_r}(V(\lambda)^{(r)},\Ext_{B_r}^2(k,L(\lambda) \otimes \lambda))
\]
is at most one-dimensional. Finally, we show that the $\Hom_{B/B_r}(V(\lambda)^{(r)},\Ext_{B_r}^2(k,L(\lambda) \otimes \lambda))$ is zero if $\lambda = \alpha_0$ and $\Phi$ has two root lengths, and is one-dimensional otherwise.

To begin, the assumptions on $p$ imply that $L(\lambda) = H^0(\lambda)$. If $\lambda = \wtalpha$, then this is true by \cite[Lemma 4.1A]{McNinch:2002}, since $H^0(\wtalpha) \cong \g$. If $\Phi$ is of type $C_n$ and $\lambda = \omega_2$, then the condition $L(\lambda) = H^0(\lambda)$ follows because the zero weight is the only dominant weight less than $\omega_2$, and the assumption that $p$ does not divide $n$ implies that $\Ext_G^1(k,L(\omega_2)) = 0$ (cf.\ \cite{Kleshchev:2001} or the reformulation in \cite[Theorem 5.2.1]{UGA:2011}), hence that the trivial module does not occur as a composition factor of $H^0(\omega_2)$. For the other Lie types and the remaining possible values for $\lambda$, the condition $L(\lambda) = H^0(\lambda)$ follows from \cite[\S 4.6]{Jantzen:1991} and \cite[II.8.21]{Jantzen:2003}. Next, if $\mu \in X(T)_+$ and $\mu \notin \set{0,\lambda}$, then $\Ext_G^2(V(\mu)^{(r)},L(\lambda) \otimes H^0(\mu)) = 0$ by Remarks \ref{remark:boundonExt2dimension}, \ref{remark:nonzeroTfpweight}, and \ref{remark:nonzeroTfqweight}, while for $\mu = 0$, $\Ext_G^2(k,L(\lambda)) = 0$ by \cite[II.4.13]{Jantzen:2003} since $L(\lambda) = H^0(\lambda)$. Then (b) in Theorem \ref{theorem:alternatecalculation} holds with $\sigma = \lambda$.

To verify (c) in Theorem \ref{theorem:alternatecalculation}, it suffices to show in each case that $S_{< \lambda}$ is either $0$ or the trivial module. This is sufficient because $\Ext_G^3(k,L(\lambda)) = 0$ by the fact $L(\lambda) = H^0(\lambda)$. Equivalently, by the description in Section \ref{subsection:alternatecalculation} of the filtration on $S_{<\lambda}$, it suffices to show that the only possible dominant weight $\nu < \lambda$ that is also linked to $\lambda$ is the zero weight. Consulting the lists in \cite[\S 7.1]{UGA:2011}\footnote{In the preprint version of \cite{UGA:2011}, the information for type $D_n$ in Section 7.1 is incorrect. The correct description for type $D_n$ is as follows: Set $\omega_0 = 0$. The weights $\omega_n$, $\omega_{n-1}$, and $\omega_1$ are minimal with respect to $\leq$. The remaining fundamental dominant weights form two independent chains, $\omega_{n-2} > \omega_{n-4} > \cdots$ and $\omega_{n-3} > \omega_{n-5} > \cdots$, with $\omega_0$ appearing as the smallest term in the chain having even indices, and $\omega_1$ appearing as the smallest term in the chain with odd indices.} (cf.\ also Remark \ref{remark:minimaldomroot}), and applying Lemmas \ref{lemma:typeBlinkage} and \ref{lemma:typeClinkage}, this is clear if $\Phi$ is not of type $F_4$ or $G_2$. For types $F_4$ and $G_2$ we have checked by explicit computer calculations using GAP \cite{GAP4} that $\wtalpha$ is not linked to $\alpha_0$ if $p > 3$. Thus, (c) holds for $\sigma = \lambda$. Finally, (d) follows from the description of the filtration on $Q_{\nleq \lambda}$, the proof of Theorem \ref{theorem:vanishingconditionforExt1iso}, and the fact $\Ext_G^1(k,L(\lambda)) = \Ext_G^1(k,H^0(\lambda)) = 0$. Then the hypotheses of Theorem \ref{theorem:alternatecalculation} are satisfied for $\sigma = \lambda$.

To finish the proof we must compute the space
\begin{equation} \label{eq:Homiso}
\Hom_{B/B_r}(V(\lambda)^{(r)},\Ext_{B_r}^2(k,L(\lambda) \otimes \lambda)) \cong \Hom_{B/B_r}(V(\lambda)^{(r)},\Ext_{U_r}^2(k,L(\lambda)) \otimes \lambda).
\end{equation}
The above identification follows from the isomorphisms $\Ext_{B_r}^2(k,L(\lambda) \otimes \lambda) \cong \Ext_{U_r}^2(k,L(\lambda) \otimes \lambda)^{T_r}$ and $\Ext_{U_r}^2(k,L(\lambda) \otimes \lambda) \cong \Ext_{U_r}^2(k,L(\lambda)) \otimes \lambda$, and the fact that any $B/B_r$-module homomorphism $V(\lambda)^{(r)} \rightarrow \Ext_{U_r}^2(k,L(\lambda) \otimes \lambda)$ automatically has image in the subspace of $T_r$-invariants. We divide the remainder of the proof into several cases, depending on the values of $r$ and $\lambda$.

{\bf The case $\mathbf{r=1}$.} Consider the spectral sequence \eqref{eq:U1Liealgspecseq} from the proof of Lemma \ref{lemma:Ext2U1Tfqinvariants}. Applying the exact functor $- \otimes \lambda$ to \eqref{eq:U1Liealgspecseq}, we obtain the new spectral sequence
\begin{equation} \label{eq:newspecseq}
E_2^{2i,j} \otimes \lambda = S^i(\fu^*)^{(1)} \otimes \opH^j(\fu,L(\lambda)) \otimes \lambda \Rightarrow \Ext_{U_1}^{2i+j}(k,L(\lambda)) \otimes \lambda.
\end{equation}
Since $E_\infty^{1,1} = E_2^{1,1} = 0$, we obtain from \eqref{eq:newspecseq} the short exact sequence of $B$-modules
\begin{equation} \label{eq:abutmentses}
0 \rightarrow E_\infty^{2,0} \otimes \lambda \rightarrow \Ext_{U_1}^2(k,L(\lambda)) \otimes \lambda \rightarrow E_\infty^{0,2} \otimes \lambda \rightarrow 0.
\end{equation}
The arguments for the case $r = 1$ now diverge depending on the value of $\lambda$.

{\bf The case $\mathbf{r=1}$, Type $\mathbf{A_2}$, $\lambda \in \set{\omega_1,\omega_2}$.} In this case $p = q = 5$. Applying the exact functor $(-)^{T_1}$ to \eqref{eq:abutmentses}, we obtain the new short exact sequence of $B/B_1$-modules
\[
0 \rightarrow (E_\infty^{2,0} \otimes \lambda)^{T_1} \rightarrow \Ext_{B_1}^2(k,L(\lambda) \otimes \lambda) \rightarrow (E_\infty^{0,2} \otimes \lambda)^{T_1} \rightarrow 0.
\]
We claim that $(E_2^{2,0} \otimes \lambda)^{T_1} = 0$, and hence that $(E_\infty^{2,0} \otimes \lambda)^{T_1} = 0$. Indeed, the weights of $E_2^{2,0} \otimes \lambda \cong (\fu^*)^{(1)} \otimes \opH^0(\fu,L(\lambda)) \otimes \lambda$ have the form $5\beta + w_0\lambda + \lambda = 5\beta \pm (\omega_1 - \omega_2)$ for $\beta \in \Phi^+$, and one can check that no weight of this form is an element of $5X(T)$. Then $\Ext_{B_1}^2(k,L(\lambda) \otimes \lambda) \cong (E_\infty^{0,2} \otimes \lambda)^{T_1}$. Next we show that $\Hom_{B/B_1}(V(\lambda)^{(1)},E_\infty^{0,2} \otimes \lambda)$ is nonzero; we verify this for the case $\lambda = \omega_1$, and leave the completely analogous case $\lambda = \omega_2$ to the reader. First, $E_2^{0,2} \cong \opH^2(\fu,L(\omega_1))$ has $T$-weights $4\omega_1$ and $4\omega_1 - \omega_2$ by \cite[Theorem 4.2.1]{UGA:2009} (cf.\ also the footnote in \cite[\S 6.1]{UGA:2011}). These weights are incomparable in $X(T)$, so we conclude that $\opH^2(\fu,L(\omega_1))$ is semisimple as a $B$-module. The space $E_\infty^{0,2}$ is a $B$-module subquotient of $E_2^{0,2}$, so it is also semisimple as a $B$-module. Then $5\omega_1$ occurs with multiplicity one in the semisimple $B$-module $E_\infty^{0,2} \otimes \omega_1$ by Remark \ref{remark:nonzeroTfpweight}. Since $V(\omega_1)^{(1)}$ is generated as a $B$-module by its $5\omega_1$-weight space, we conclude that $\Hom_{B/B_1}(V(\omega_1)^{(1)}, E_\infty^{0,2} \otimes \lambda) \cong k$.

{\bf The case $\mathbf{r=1}$, $\lambda \in \set{\alpha_0,\wtalpha}$.} It follows from Remark \ref{remark:nonzeroTfpweight} that $p\lambda$ occurs once as a weight of $E_\infty^{2,0} \otimes \lambda$, but does not occur as a weight of $E_\infty^{0,2} \otimes \lambda$. Then $\Hom_{B/B_1}(V(\lambda)^{(1)},E_\infty^{0,2} \otimes \lambda) = 0$, so from \eqref{eq:abutmentses} we obtain
\[
\Hom_{B/B_1}(V(\lambda)^{(1)},E_\infty^{2,0} \otimes \lambda) \cong \Hom_{B/B_1}(V(\lambda)^{(1)},\Ext_{U_1}^2(k,L(\lambda)) \otimes \lambda).
\]

We claim that $E_\infty^{2,0} \cong E_2^{2,0}$, or equivalently, that the differential $d_2: E_2^{0,1} \rightarrow E_2^{2,0}$ is the zero map. To prove the claim, we show that no weight of $E_2^{2,0} \otimes \lambda$ is a weight of $E_2^{0,1} \otimes \lambda$, which implies that the differential is the zero map. Since $\lambda$ is a restricted weight, one has $\opH^0(\fu,L(\lambda)) = L(\lambda)^{\fu} = L(\lambda)_{w_0\lambda}$, so the weights of $E_2^{2,0} \otimes \lambda \cong (\fu^*)^{(1)} \otimes \opH^0(\fu,L(\lambda)) \otimes \lambda$ have the form $p\beta + w_0 \lambda + \lambda = p\beta$ for $\beta \in \Phi^+$. In particular, the weights of $E_2^{2,0} \otimes \lambda$ are all nonzero $T_1$-invariant weights. Now we will show that the zero weight is the only $T_1$-invariant weight in $E_2^{0,1} \otimes \lambda$, hence that no weight of $E_2^{2,0} \otimes \lambda$ is a weight of $E_2^{0,1} \otimes \lambda$. By \cite[\S 2.5]{UGA:2009} and \cite[Proposition 2.2]{Bendel:2007}, the weights of $E_2^{0,1} \otimes \lambda \cong \opH^1(\fu,L(\lambda)) \otimes \lambda$ have the form $\alpha + \nu + \lambda$ for $\alpha \in \Delta$ and $\nu$ a weight of $L(\lambda)$. Since $\lambda \in \set{\alpha_0,\wtalpha}$, one has $\nu \in \Phi \cup \set{0}$. Suppose $\alpha + \nu + \lambda = p \sigma$ for some $\sigma \in X(T)$. For $\gamma \in \Delta$, one has
\[
p \abs{(\sigma,\gamma^\vee)} \leq \abs{(\alpha,\gamma^\vee)} + \abs{(\nu,\gamma^\vee)} + (\lambda,\gamma^\vee) \leq \begin{cases} 7 & \text{if $\Phi$ is of type $G_2$,} \\ 6 & \text{otherwise,} \end{cases}
\]
because $\abs{(\alpha,\gamma^\vee)}$ and $\abs{(\nu,\gamma^\vee)}$ are each at most $2$ (resp.\ $3$ if $\Phi$ is of type $G_2$), and $(\lambda,\gamma^\vee)$ is at most $2$ (resp.\ $1$ if $\Phi$ is of type $G_2$). Thus, if $p > 5$ (resp.\ $p > 7$ for type $G_2$), then $(\sigma,\gamma^\vee) = 0$ for all $\gamma \in \Delta$, and hence $\sigma = 0$. It remains to show that $\sigma = 0$ if $p = 7$ and $\Phi$ is of type $G_2$, or if $p=5$ and $\Phi$ is of type $A_n$, $B_n$, $D_n$, or $E_6$. We give the argument for types $A_n$, $B_n$, and $D_n$, and leave the details for types $E_6$ and $G_2$ to the reader, since those types can be handled similarly. Observe that $\sigma \in \Z\Phi$, since $\alpha + \nu + \lambda \in \Z\Phi$ and $p \nmid [X(T):\Z\Phi]$. Write $\alpha + \nu = \sum_{\delta \in \Delta} c_\delta \delta$, and $\alpha + \nu + \lambda = \sum_{\delta \in \Delta} d_\delta \delta$ with $c_\delta,d_\delta \in \Z$. If $\nu \leq 0$, then $-2 \leq c_\delta \leq 1$ and $-1 \leq d_\delta \leq 3$. In particular, if $\nu \leq 0$, then $\alpha + \nu + \lambda \in 5\Z\Phi$ only if $d_\delta = 0$ for all $\delta \in \Delta$, that is, only if $\sigma = 0$. On the other hand, if $\nu \geq 0$, then $0 \leq c_\delta \leq 3$, and $1 \leq d_\delta \leq 5$. Moreover, $d_\delta = 5$ only if $\delta = \alpha$. Since $\alpha + \nu + \lambda = 5\alpha$ if and only if $\nu + \lambda = 4\alpha$, which is absurd, we conclude that $\alpha + \nu + \lambda \notin 5\Z\Phi$. This completes the proof of the claim that no weight from $E_2^{2,0} \otimes \lambda$ is a weight of $E_2^{0,1} \otimes \lambda$, and hence that $E_\infty^{2,0} \cong E_2^{2,0}$.

We have shown that $E_\infty^{2,0} \otimes \lambda \cong E_2^{2,0} \otimes \lambda \cong (\fu^*)^{(1)} \otimes w_0\lambda \otimes \lambda \cong (\fu^*)^{(1)}$ as $B$-modules. Then
\[
\Hom_{B/B_1}(V(\lambda)^{(1)},\Ext_{U_1}^2(k,L(\lambda)) \otimes \lambda) \cong \Hom_B(V(\lambda),\fu^*) \cong \Hom_G(V(\lambda),\ind_B^G \fu^*).
\]
We have $\ind_B^G \fu^* \cong \opH^2(G_1,k)^{(-1)}$ by \cite[Theorem 6.2]{Bendel:2007}, and $\opH^2(G_1,k)^{(-1)} \cong \g^* \cong H^0(\wtalpha)$ by \cite[Lemma 3.11]{Friedlander:1983}. Then $\Hom_{B/B_1}(V(\lambda)^{(1)},\Ext_{U_1}^2(k,L(\lambda)) \otimes \lambda) \cong \Hom_G(V(\lambda),H^0(\wtalpha))$ is zero if $\lambda = \alpha_0 \neq \wtalpha$, and is one-dimensional otherwise by \cite[II.4.13]{Jantzen:2003}.

{\bf The case $\mathbf{r > 1}$.} In this case again $\lambda \in \set{\alpha_0,\wtalpha}$. Consider the spectral sequence \eqref{eq:LHSforUrTfqinvariants} from the proof of Lemma \ref{lemma:Ext2UrTfqinvariants}. As in the case $r=1$, we apply the exact functor $- \otimes \lambda$ to obtain the new spectral sequence
\[
E_2^{i,j} \otimes \lambda \cong \Ext_{U_r/U_1}^i(k,\Ext_{U_1}^j(k,L(\lambda)) \otimes \lambda) \Rightarrow \Ext_{U_r}^{i+j}(k,L(\lambda)) \otimes \lambda.
\]
It follows from the proof of Lemma \ref{lemma:Ext2UrTfqinvariants} and from Remark \ref{remark:nonzeroTfqweight} that the weight $p^r \lambda$ occurs with multiplicity one in $\Ext_{U_r}^2(k,L(\lambda) \otimes \lambda)$. Specifically, the weight $p^r \lambda$ occurs in the $B$-submodule $E_\infty^{2,0} \otimes \lambda$ of $\Ext_{U_r}^2(k,L(\lambda)) \otimes \lambda$ arising from the filtration on the abutment of the spectral sequence. Then any $B/B_r$-module homomorphism $V(\lambda)^{(r)} \rightarrow \Ext_{U_r}^2(k,L(\lambda)) \otimes \lambda$ will necessarily have image in the submodule $E_\infty^{2,0} \otimes \lambda$ of $\Ext_{U_r}^2(k,L(\lambda)) \otimes \lambda$.

Observe that $E_2^{0,1} = \Ext_{U_1}^1(k,L(\lambda))^{U_r/U_1}$ is $B$-semisimple by Theorem \ref{theorem:semisimplicity}. Then there exists a semisimple $B$-submodule $Q$ of $\Ext_{U_1}^1(k,L(\lambda))$ such that the following sequence is exact:
\[
0 \rightarrow Q \otimes \lambda \stackrel{d_2}{\rightarrow} E_2^{2,0} \otimes \lambda \rightarrow E_\infty^{2,0} \otimes \lambda \rightarrow 0.
\]
As in the case $r=1$, we apply the exact functor $(-)^{T_r}$ to obtain the new short exact sequence
\begin{equation} \label{eq:Trinvariantsses}
0 \rightarrow (Q \otimes \lambda)^{T_r} \rightarrow (E_2^{2,0} \otimes \lambda)^{T_r} \rightarrow (E_\infty^{2,0} \otimes \lambda)^{T_r} \rightarrow 0.
\end{equation}
Now the weights of $(Q \otimes \lambda)^{T_r}$ are weights of $(\Ext_{U_r}^1(k,L(\lambda)) \otimes \lambda)^{T_r} \cong \Ext_{B_r}^1(k,L(\lambda) \otimes \lambda)$, so it follows from the proof of Lemma \ref{lemma:HomtoExt2vanish} that all weights of $(Q \otimes \lambda)^{T_r}$ are dominant.\footnote{The first part of the proof of Lemma \ref{lemma:HomtoExt2vanish}, where it is established that the weights of $\Ext_{B_r}^1(k,L(\lambda) \otimes \mu)$ are all dominant, only requires that $\lambda \in X_r(T)$, $\mu \in X(T)_+$, and that $p^r$ be large relative to the weights of $\Ext_{U_r}^1(k,L(\lambda))$. These conditions are satisfied in the present situation by our assumptions on $p$ and $q$.} Then $(\ind_{B/B_r}^{G/G_r} (Q \otimes \lambda)^{T_r})^{(-r)}$ admits a good filtration (in fact, is a direct sum of induced modules), and $R^i \ind_{B/B_r}^{G/G_r} (Q \otimes \lambda)^{T_r} = 0$ for all $i > 0$ by Kempf's Vanishing Theorem, so it follows from \cite[I.4.5, II.4.16]{Jantzen:2003} that
\[
\Ext_{B/B_r}^1(V(\lambda)^{(r)},(Q \otimes \lambda)^{T_r}) \cong \Ext_{G/G_r}^1(V(\lambda)^{(r)},\ind_{B/B_r}^{G/G_r} (Q \otimes \lambda)^{T_r}) = 0.
\]
Consequently, by considering the long exact sequence in cohomology associated to the short exact sequence \eqref{eq:Trinvariantsses} and the functor $\Hom_{B/B_r}(V(\lambda)^{(r)},-)$, one sees that $\Hom_{B/B_r}(V(\lambda)^{(r)},(E_2^{2,0} \otimes \lambda)^{T_r})$ surjects onto $\Hom_{B/B_r}(V(\lambda)^{(r)},(E_\infty^{2,0} \otimes \lambda)^{T_r}) \cong \Hom_{B/B_r}(V(\lambda)^{(r)},\Ext_{U_r}^2(k,L(\lambda)) \otimes \lambda)$. Now
\[
(E_2^{2,0} \otimes \lambda)^{T_r} \cong \Ext_{B_r}^2(k,\Hom_{U_1}(k,L(\lambda)) \otimes \lambda) \cong \Ext_{B_r}^2(k,w_0\lambda \otimes \lambda) \cong \opH^2(B_r,k) \cong (\fu^*)^{(r)}
\]
 by \cite[Proposition 5.4 and Theorem 5.3]{Bendel:2007}. Then it follows as in the case $r=1$ that
\[
\Hom_{B/B_r}(V(\lambda)^{(r)},(E_2^{2,0} \otimes \lambda)^{T_r}) \cong \Hom_B(V(\lambda),\fu^*) \cong \Hom_G(V(\lambda),\ind_B^G \fu^*)
\]
is zero if $\lambda = \alpha_0 \neq \wtalpha$, and is one-dimensional otherwise. This completes the case $r > 1$, and thus completes the proof.
\end{proof}

\section{Partial results for Type C} \label{section:typeC}

\subsection{Reduction to small primes} \label{subsection:TypeCreduction}

Let $G$ be a simple, simply-connected algebraic group with underlying root system of type $C_n$ with $n \geq 3$, and let $\lambda = \omega_j$ be a fundamental dominant weight with $j \neq 2$. If $p > 3$, then $\opH^2(G,L(\omega_j)) \cong \opH^2(\Gfq,L(\omega_j))$ by Theorem \ref{theorem:Ext2restrictioniso} provided $q > 5$ if $j$ is even. If $j$ is odd, then $\omega_j \notin \Z\Phi$, and $\opH^2(G,L(\omega_j)) = 0$ by the Linkage Principle. So for the remainder of this section assume that $j$ is even. If $p > h = 2n$, then we still have $\opH^2(G,L(\omega_j)) = 0$ by Lemma \ref{lemma:largeprimevanishing}. In fact, it turns out that $\opH^2(G,L(\omega_j)) = 0$ under the weaker assumption $p > n$.

\begin{lemma} \label{lemma:typeclargeprimevanishing}
Let $G$ be a simple algebraic group of type $C_n$ with $n \geq 3$. Let $\omega_j$ be a fundamental dominant weight with $j$ even. Suppose $p > n$. Then $\omega_j$ is not linked to $0$ under the dot action of the extended affine Weyl group $\wh{W}_p = W \ltimes pX(T)$. In particular, $\omega_j$ is not linked to zero under the dot action of the ordinary affine Weyl group $W_p$, and consequently, $\opH^2(G,L(\omega_j)) = 0$.
\end{lemma}

\begin{proof}
The proof is by contradiction. Suppose $p > n$ and $\omega_j \in \wh{W}_p \cdot 0$. Then there exist $w \in W$ and $\nu \in X(T)$ such that $w(\rho) = \omega_j + \rho + p\nu$, where $\rho = \sum_{i=1}^n \omega_i$. In terms of the $\ve$-basis for $\Phi$, we can write $\omega_j = \sum_{i=1}^j \ve_i$, $\rho = \sum_{i=1}^n (n-i+1)\ve_i$, and $\nu = \sum_{i=1}^n \nu_i \ve_i$, for some integers $\nu_i \in \Z$. As in type $B_n$, $W$ acts as the group of all permutations and sign changes of the vectors $\ve_1,\ldots,\ve_n$. Then the equation $w(\rho) = \omega_j + \rho + p\nu$ may be rewritten as $w(\rho) = \sum_{i=1}^n b_i \ve_i$, where
\[
b_i = \begin{cases} n-i+2 + p\nu_i & \text{if $1 \leq i \leq j$} \\ n-i+1+p\nu_i & \text{if $j < i \leq n$}. \end{cases}
\]
Moreover, up to sign changes, the set $\set{b_1,\ldots,b_n}$ must coincide with the set $\set{1,\ldots,n}$ of coefficients for $\rho$. Then $(\prod_{i=1}^n b_i) \equiv (\pm n!) \mod p$, that is, $(n+1)(n) \cdots (n-j+2)(n-j) \cdots (2)(1) \equiv (\pm n!) \mod p$. Since $p > n$, we can cancel like terms on either side of the congruence equation to obtain $(n+1) \equiv \pm (n-j+1) \mod p$. If $(n+1) \equiv (n-j+1) \mod p$, then $j \equiv 0 \mod p$, a contradiction because $2 \leq j \leq n < p$. On the other hand, if $(n+1) \equiv -(n-j+1) \mod p$, then $2n+2-j \equiv 0 \mod p$, which implies since $j$ is even and $p \neq 2$ that $(n+1 - \frac{j}{2}) \equiv 0 \mod p$. This is again a contradiction, because $0 < (n+1-\frac{j}{2}) \leq n < p$. Thus, we conclude if $p > n$ that $\omega_j \notin \wh{W}_p \cdot 0$, and hence that $\opH^2(G,L(\omega_j)) = 0$ by the Linkage Principle.
\end{proof}

\subsection{Homomorphisms in the truncated category} \label{subsection:truncatedcategory}

We next describe how the problem of computing the rational cohomology groups $\opH^2(G,L(\lambda))$ can be reduced to that of computing a certain space of $G$-homomorphisms in a suitable truncated category of rational $G$-modules. This setup is valid for arbitrary simple algebraic groups, so we will describe the results first in that generality, and then specialize to fundamental dominant weights in type $C_n$.

Let $\lambda \in X(T)_+$, and set $\Gamma = \set{\mu \in X(T)_+: \mu \leq \lambda}$. Then $\Gamma$ is a saturated subset of $X(T)_+$, and we can consider the full subcategory $\cC(\Gamma)$ of all rational $G$-modules with composition factors having highest weights in $\Gamma$. This is the \emph{truncated category} associated to $\Gamma$. If $V$ and $V'$ are modules in $\cC(\Gamma)$, then $\Ext_{\cC(\Gamma)}^i(V,V') \cong \Ext_G^i(V,V')$ for all $i \geq 0$ \cite[II.A.10]{Jantzen:2003}. It is well-known that $\cC(\Gamma)$ is a highest weight category \cite[Example 3.3(d) and Theorem 3.5(a)]{Cline:1988}, and hence Morita equivalent to the module category of a finite-dimensional quasi-hereditary algebra \cite[Theorem 3.6]{Cline:1988}.

Suppose $\lambda \geq 0$, so that $0 \in \Gamma$ and hence $k \in \cC(\Gamma)$. Let $P_0$ be the projective cover in $\cC(\Gamma)$ of $k$, and let $\Omega^1(k)$ be the kernel of the surjective map $P_0 \rightarrow k$. Then $P_0$ admits a Weyl filtration that restricts to a Weyl filtration on $\Omega^1(k)$; cf.\ \cite[Definition 3.1]{Cline:1988}. Since $\Omega^1(k)$ admits a Weyl filtration, the short exact sequence
\begin{equation} \label{eq:sesLlambda}
0 \rightarrow L(\lambda) \rightarrow H^0(\lambda) \rightarrow M \rightarrow 0
\end{equation}
gives rise via the associated long exact sequence in cohomology to the four-term exact sequence
\[
0 \rightarrow \Hom_G(\Omega^1(k),L(\lambda)) \rightarrow \Hom_G(\Omega^1(k),H^0(\lambda)) \rightarrow \Hom_G(\Omega^1(k),M) \rightarrow \Ext_G^1(\Omega^1(k),L(\lambda)) \rightarrow 0.
\]
Moreover, if $\lambda \neq 0$, then the short exact sequence $0 \rightarrow \Omega^1(k) \rightarrow P_0 \rightarrow k \rightarrow 0$ gives rise via the associated long exact sequence  in cohomology and \cite[II.4.13]{Jantzen:2003} to the isomorphism $\Hom_G(P_0,H^0(\lambda)) \cong \Hom_G(\Omega^1(k),H^0(\lambda))$, so we can identify $\Hom_G(\Omega^1(k),H^0(\lambda))$ as a vector space with $k^m$ where $m = [H^0(\lambda):k]$, and thus rewrite the four-term exact sequence as
\begin{equation} \label{eq:fourterm}
0 \rightarrow \opH^1(G,L(\lambda)) \rightarrow k^m \rightarrow \Hom_G(\Omega^1(k),M) \rightarrow \opH^2(G,L(\lambda)) \rightarrow 0.
\end{equation}

Now specialize to the case where $G$ is a simple algebraic group of type $C_n$ with $n \geq 3$, and $\lambda = \omega_j$ for some even integer $1 \leq j \leq n$. Then by the list in \cite[\S 7.1]{UGA:2011}, $\Gamma = \set{ \mu \in X(T)_+: \mu \leq \lambda} = \set{\omega_t: \text{$t$ is even and $t \leq j$}} \cup \set{0}$. In this case, the dimension of $\opH^1(G,L(\omega_j))$ can be computed from \cite[Theorem 3.5(iii)]{Kleshchev:2001}, while the value for $m = [H^0(\omega_j):k]$ can be computed from the explicit description for the submodule lattice of $V(\omega_j) = H^0(\omega_j)^*$ given in \cite[\S 3]{Kleshchev:1999}. Thus, for fundamental dominant weights in the root lattice in type $C_n$, the computation of $\opH^2(G,L(\omega_j))$ can be reduced to understanding the dimension of the $\Hom$-space
\[
\Hom_G(\Omega^1(k),M) = \Hom_G(\Omega^1(k),H^0(\omega_j)/L(\omega_j)).
\]

\subsection{Example; Type \texorpdfstring{$C_{12}$, $p=3$}{C12, p=3}}

Let $G$ be a simple algebraic group with underlying root system of type $C_{12}$, and suppose $p=3$. Let $\Gamma = \set{\mu \in X(T)_+ : \mu \leq \omega_{12}} = \set{0,\omega_2,\omega_4,\omega_6,\omega_8,\omega_{10},\omega_{12}}$. The structures of the nontrivial Weyl modules $V(\omega_{2i})$ for $1 \leq i \leq 6$, as computed from the algorithm presented in \cite[\S 3]{Kleshchev:1999}, are represented in Figure \ref{figure:inducedmodules}. The structures of the  induced modules $H^{0}(\omega_{2i})$ are obtained by flipping the diagrams for the corresponding Weyl modules upside down. As before, let $P_0$ be the projective cover of the trivial module in $\cC(\Gamma)$. Then the structure of $P_0$ is represented in Figure \ref{figure:P0}. The structure of $\Omega^1(k)$ is obtained from that of $P_0$ by deleting the topmost node  $L(\omega_{0})$. As usual, the diagrams should be read from top to bottom, with the head of the module at the top, and the socle at the bottom. Given a particular node $L$ in one of the diagrams, the module represented by that diagram then contains a submodule having all nodes below or equal to $L$ as composition factors. So for example, $P_0$ contains a submodule isomorphic to $V(\omega_{8})$, with composition factors $L(\omega_8)$, $L(\omega_6)$, and $L(\omega_0)$.

By \cite[Theorem 3.5]{Kleshchev:2001}, if $L$ and $L'$ are simple modules in $\cC(\Gamma)$, then $\dim \Ext_G^1(L,L') \leq 1$. From this and from Figures \ref{figure:inducedmodules} and \ref{figure:P0} we deduce that $H^0(\omega_6)/L(\omega_6)$ occurs as a quotient of $\Omega^1(k)$. (The one-dimensional upper bound on the $\Ext^1$-groups between simple modules is necessary to conclude that the subquotient of $P_0$ that looks like $H^0(\omega_6)/L(\omega_6)$ actually is isomorphic to $H^0(\omega_6)/L(\omega_6)$, and is not some other inequivalent extension of $L(\omega_2)$ by $L(\omega_0)$.) Also by \cite[Theorem 3.5]{Kleshchev:2001} (or from Figure \ref{figure:inducedmodules}) we get $\opH^1(G,L(\omega_6)) \cong k$ and $[H^0(\omega_6):k] = 1$. Then the four-term exact sequence \eqref{eq:fourterm} implies that $\opH^2(G,L(\omega_6)) \cong \Hom_G(\Omega^1(k),H^0(\omega_6)/L(\omega_6)) \cong k$.

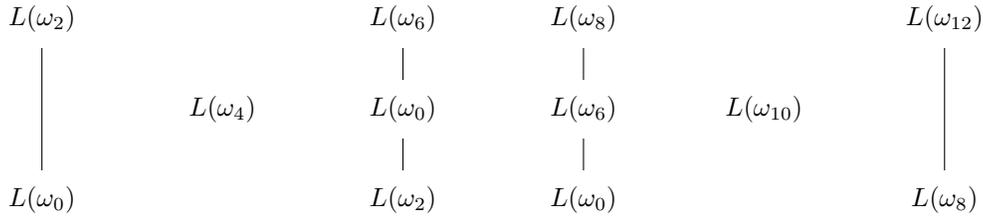
\begin{figure}[htbp]
\tikzstyle{every node} = [fill=white, minimum size=2em, font=\small]
\newcommand{\cf}[1]{$L(\omega_{#1})$}
\begin{tikzpicture}[scale=1.2,x=2cm]
\draw (0,1) node {\cf{2}} -- (0,-1) node {\cf{0}};
\draw (1,0) node {\cf{4}};
\draw (2,1) node {\cf{6}} -- ++ (0,-1) node {\cf{0}} -- ++ (0,-1) node {\cf{2}};
\draw (3,1) node {\cf{8}} -- ++ (0,-1) node {\cf{6}} -- ++ (0,-1) node {\cf{0}};
\draw (4,0) node {\cf{10}};
\draw (5,1) node {\cf{12}} -- ++ (0,-2) node {\cf{8}};
\end{tikzpicture}
\caption{Weyl modules $V(\omega_{2i})$, $1 \leq i \leq 6$, type $C_{12}$, $p=3$.}
\label{figure:inducedmodules}
\end{figure}

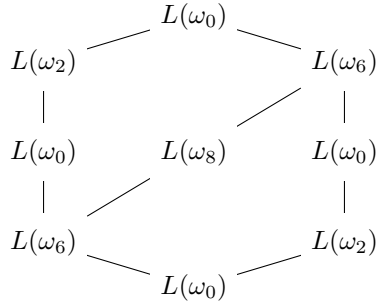
\begin{figure}[htbp]
\tikzstyle{every node} = [fill=white, minimum size=2em, font=\small]
\newcommand{\cf}[1]{$L(\omega_{#1})$}
\begin{tikzpicture}[x=2cm,y=1.2cm]
  \draw (0,0) node {\cf{0}}
  -- ++ (1,-.5) node (ur) {\cf{6}}
  -- ++ (0,-1) node {\cf{0}}
  -- ++ (0,-1) node {\cf{2}}
  -- ++ (-1,-.5) node {\cf{0}}
  -- ++ (-1, .5) node (ll) {\cf{6}}
  -- ++ (0,1) node {\cf{0}}
  -- ++ (0,1) node {\cf{2}}
  -- cycle;
  
  \draw (ll) -- ++ (1,1) node {\cf{8}} -- (ur);
\end{tikzpicture}
\caption{Projective cover of $k$ in $\cC(\{\mu \leq \omega_{12}\})$, type $C_{12}$, $p=3$.} \label{figure:P0}
\end{figure}

\subsection{Examples for small primes and low ranks}

Continue to assume that $G$ is a simple algebraic group of type $C_n$ with $n \geq 3$. As the rank $n$ increases, it becomes increasingly difficult to compute the structure of the module $\Omega^1(k)$ as in Section \ref{subsection:truncatedcategory}. Thus, unless $n$ is very small, it is often more practical to apply ad hoc methods to compute the dimension of the second cohomology group $\opH^2(G,L(\omega_j))$. For example, by considering a short exact sequence of the form \eqref{eq:sesLlambda} and applying \cite[II.4.13]{Jantzen:2003}, one obtains $\opH^2(G,L(\lambda)) \cong \opH^1(G,M)$. One can then use Adamovich's \cite{Adamovich:1986} combinatorial description for the submodule structure of the induced modules $H^0(\omega_i)$ (rather, for the dual Weyl modules $V(\omega_i) = H^0(\omega_i)^*$), as presented in \cite{Kleshchev:1999}, together with the results in \cite{Kleshchev:2001}, to inductively compute the dimension of $\opH^1(G,M)$.

We illustrate the above approach for the case $\Phi = C_{12}$ with $p=3$. For $j \in \set{4,10}$ one has $L(\omega_j) = H^0(\omega_j)$, so $\opH^2(G,L(\omega_j)) = 0$. For $j = 2$ one obtains $\opH^2(G,L(\omega_2)) \cong \opH^1(G,k) = 0$, and for $j=6$ one obtains $\opH^2(G,L(\omega_6)) \cong \opH^1(G,H^0(\omega_6)/L(\omega_6))$, and the short exact sequence
\[
0 \rightarrow k \rightarrow H^0(\omega_6)/L(\omega_6) \rightarrow L(\omega_2) \rightarrow 0.
\]
Since $\opH^1(G,k) = \opH^2(G,k) = 0$, one obtains from the associated long exact sequence in cohomology that $\opH^1(G,H^0(\omega_6)/L(\omega_6)) \cong \opH^1(G,L(\omega_2)) = k$. Next, for $j=8$ one has $\opH^2(G,L(\omega_8)) \cong \opH^1(G,H^0(\omega_8)/L(\omega_8))$ and the short exact sequence
\[
0 \rightarrow H^0(\omega_8)/L(\omega_8) \rightarrow H^0(\omega_6) \rightarrow L(\omega_2).
\]
Since $\Hom_G(k,L(\omega_2)) = 0 = \opH^1(G,H^0(\omega_6))$, one obtains via the associated long exact sequence in cohomology that $\opH^1(G,H^0(\omega_8)/L(\omega_8)) = 0$. Finally, $\opH^2(G,L(\omega_{12})) \cong \opH^1(G,L(\omega_8)) = 0$.

A summary of our results for certain values for $n$ and for the primes $p=3$ and $p=5$ is presented in Tables \ref{table:typecp=3} and \ref{table:typecp=5}. In every example we computed, we found $\dim \opH^2(G,L(\omega_j)) \leq 1$. A point of interest in our calculations relates to the Lusztig Conjecture for the characters of irreducible rational $G$-modules. The Lusztig Conjecture is equivalent to the statement that the rational cohomology group $\Ext_{G}^{i}(L(\lambda), H^{0}(\mu))$ is nonzero only if the degree $i$ has a fixed parity depending on $\lambda$ and $\mu$ \cite[II.C.2]{Jantzen:2003}. Our results show that this parity vanishing condition does not hold for the primes $3$ and $5$.  For example, with $p=3$ and $n=12$,
\begin{equation} 
\text{Ext}_{G}^{1}(L(\omega_{6}),H^{0}(0))\cong \opH^1(G,L(\omega_6))\cong k
\end{equation} 
and 
\begin{equation} 
\text{Ext}_{G}^{2}(L(\omega_{6}),H^{0}(0))\cong \opH^2(G,L(\omega_6)) \cong k. 
\end{equation} 
When $p=5$ and $n=30$, the parity vanishing condition also fails to hold for $L(\omega_{10})$; see Table \ref{table:typecp=5}.

\begin{table}[htbp]
  \newcommand{\none}{\textrm{none}}
  \renewcommand{\arraystretch}{1.2}
 \begin{tabular}{c|l}
    $n$ & $j$ \\
    \hline
    6 & 6\\
    7 & 6\\
    8 & \none \\
    9 & 6\\
    10 & 6\\
    11 & \none \\
    12 & 6\\
    13 & 6\\
    14 & \none \\
  \end{tabular}
  \hspace{3ex}
  \begin{tabular}{c|l}
    $n$ & $j$ \\
    \hline
    15 & 6, 8\\
    16 & 6, 10\\
    17 & \none \\
    18 & 6, 14\\
    19 & 6, 16\\
    20 & 18\\
    21 & 6, 18\\
    22 & 6, 18\\
    23 & 18\\
  \end{tabular}
  \hspace{3ex}
  \begin{tabular}{c|l}
    $n$ & $j$ \\
    \hline
    24 & 6, 8, 18\\
    25 & 6, 10, 18\\
    26 & \none \\
    27 & 6, 14\\
    28 & 6, 16\\
    29 & 18\\
    30 & 6, 18\\
    31 & 6, 18\\
    32 & 18\\
  \end{tabular}
  \hspace{3ex}
  \begin{tabular}{c|l}
    $n$ & $j$ \\
    \hline
    33 & 6, 8, 18\\
    34 & 6, 10, 18\\
    35 & \none \\
    36 & 6, 14\\
    37 & 6, 16\\
    38 & 18\\
    39 & 6, 18, 20\\
    40 & 6, 18, 22\\
     & \\
\end{tabular}
\vspace{1pc}
\caption{Type $C_n$, $p = 3$.  Values of $j$ for which $\opH^2(G,L(\omega_j)) \cong k$.} \label{table:typecp=3}
\end{table}

\begin{table}[htbp]
  \newcommand{\none}{\textrm{none}}
  \renewcommand{\arraystretch}{1.2}
  \begin{tabular}{c|l}
    $n$ & $j$ \\
    \hline
    10 & 10\\
    11 & 10\\
    12 & 10\\
    13 & 10\\
    14 & \none \\
    15 & 10\\
    16 & 10\\
    17 & 10\\
    18 & 10\\
    19 & \none \\
  \end{tabular}
  \hspace{3ex}
  \begin{tabular}{c|l}
    $n$ & $j$ \\
    \hline
    20 & 10\\
    21 & 10\\
    22 & 10\\
    23 & 10\\
    24 & \none \\
    25 & 10\\
    26 & 10\\
    27 & 10\\
    28 & 10\\
    29 & \none \\
  \end{tabular}
  \hspace{3ex}
  \begin{tabular}{c|l}
    $n$ & $j$ \\
    \hline
    30 & 10\\
    31 & 10\\
    32 & 10\\
    33 & 10\\
    34 & \none \\
    35 & 10, 12\\
    36 & 10, 14\\
    37 & 10, 16\\
    38 & 10, 18\\
    39 & \none \\
  \end{tabular}
  \hspace{3ex}
  \begin{tabular}{c|l}
    $n$ & $j$ \\
    \hline
    40 & 10, 22\\
    41 & 10, 24\\
    42 & 10, 26\\
    43 & 10, 28\\
    44 & \none \\
    45 & 10, 32\\
    46 & 10, 34\\
    47 & 10, 36\\
    48 & 10, 38\\
    49 & \none \\
  \end{tabular}
  \hspace{3ex}
  \begin{tabular}{c|l}
    $n$ & $j$ \\
    \hline
    50 & 10, 42\\
    51 & 10, 44\\
    52 & 10, 46\\
    53 & 10, 48\\
    54 & 50\\
    & \\
    & \\
    & \\
    & \\
    & \\
 \end{tabular}
\vspace{1pc}
\caption{Type $C_n$, $p = 5$.  Values of $j$ for which $\opH^2(G,L(\omega_j)) \cong k$.}
\label{table:typecp=5}
\end{table}

\subsection{Open questions} \label{subsection:openquestions}

We present here some open questions and directions for future research based on the calculations and methods of this paper. First, many of the results and calculations obtained in this paper presumably hold under weaker conditions on $p$ and $q$ than those assumed in Table~\ref{table:pandq}. One would thus like to obtain sharp lower bounds on $p$ and $q$ under which results like Theorem \ref{theorem:Ext2restrictioniso}, Corollary \ref{corollary:H2vanishing}, or Theorem \ref{theorem:otherH2calculations} hold. In particular, through a more subtle analysis of the case $p = 3$, one could hope to extend by our methods the conclusion of Theorem \ref{theorem:otherH2calculations} to include Bell's calculations \cite{Bell:1978} of nonzero second cohomology for finite special linear groups over fields of characteristic three. We think it should also be possible to obtain by our methods many of the results of Kleshchev \cite{Kleshchev:1994} on the cohomology of finite Chevalley groups with coefficients in modules having one-dimensional weight spaces.

The results presented in Tables \ref{table:typecp=3} and \ref{table:typecp=5} suggest the following questions for type $C_n$:
\begin{itemize}
\item Does one always have $\dim \opH^2(G,L(\omega_j)) \leq 1$? More generally, is there a uniform upper bound, independent of $j$ or $n$, for $\dim \opH^2(G,L(\omega_j))$?
\item Given $p$ and $n$, for which $j$ is $\opH^2(G,L(\omega_j)) \neq 0$? Can the non-vanishing of $\opH^2(G,L(\omega_j))$ be described in terms of the $p$-adic decompositions of $n$ and $j$?
\end{itemize}

\section{VIGRE Algebra Group at the University of Georgia}

This project was initiated during Fall Semester 2010 under the Vertical Integration of Research and Education (VIGRE) Program sponsored by the National Science Foundation (NSF) at the Department of Mathematics of the University of Georgia (UGA). We would like to acknowledge the NSF VIGRE grant DMS-0738586 for its financial support of the project. In academic year 2010--2011, the VIGRE Algebra Group at UGA consisted of 5 faculty members, 2 postdoctoral fellows, and 9 graduate students. The email addresses of the group members are given below.

\bigskip
\begin{tabbing}
\hspace*{\parindent}\=Christopher M. Drupieski\ \ \ \= \kill
Faculty: \\[3pt]
\>Brian D.\ Boe \>brian@math.uga.edu \\
\>Jon F.\ Carlson \> jfc@math.uga.edu \\
\>Leonard Chastkofsky \> lenny@math.uga.edu \\
\>Daniel K.\ Nakano \> nakano@math.uga.edu \\
\>Lisa Townsley \> townsley@math.uga.edu\\[6pt]
Postdoctoral Fellows:  \\[3pt]
\>Christopher M.\ Drupieski \> cdrup@math.uga.edu \\
\>Niles Johnson \> njohnson@math.uga.edu \\[6pt]
Graduate Students:  \\[3pt]
\>Brian Bonsignore \> bbonsign@math.uga.edu \\
\>Theresa Brons \> tbrons@math.uga.edu \\
\>Wenjing Li \> wli@math.uga.edu \\
\>Phong Thanh Luu \> pluu@math.uga.edu \\
\>Tiago Macedo \> t025318@dac.unicamp.br \\
\>Nham Vo Ngo \> nngo@math.uga.edu \\
\>Brandon L.\ Samples \> bsamples@math.uga.edu \\
\>Andrew J.\ Talian \> atalian@math.uga.edu \\
\>Benjamin J.\ Wyser \> bwyser@math.uga.edu \\
\end{tabbing}


\providecommand{\bysame}{\leavevmode\hbox to3em{\hrulefill}\thinspace}

\end{document}